\documentclass[11pt,a4paper]{article}

\usepackage{mathrsfs} 

\usepackage{graphicx}        

\usepackage[a4paper, hmargin={2.5cm,2.5cm},vmargin={3.3cm,3.3cm}]{geometry}
\usepackage{amsmath, amssymb, amsfonts, amsbsy, latexsym, color,dsfont}
\usepackage{amscd,amsxtra}
\usepackage{amsthm}
\usepackage[normalem]{ulem}

\usepackage[T1]{fontenc}
\usepackage[latin1]{inputenc}
\usepackage[english]{babel}
\usepackage{cite}
\usepackage{paralist}
\usepackage{url}
\usepackage{twoopt}

\usepackage[bf,compact,pagestyles,small]{titlesec}
\titlespacing*{\section}{0pt}{14pt}{4pt}
\titlespacing*{\subsection}{0pt}{8pt}{3pt}

\usepackage{etoolbox}
\makeatletter
\patchcmd{\ttlh@hang}{\parindent\z@}{\parindent\z@\leavevmode}{}{}
\patchcmd{\ttlh@hang}{\noindent}{}{}{}
\makeatother

\usepackage{mathdots}

\usepackage{fancyhdr,lastpage}
\def\maketimestamp{\count255=\time
\divide\count255 by 60\relax
\edef\thetime{\the\count255:}%
\multiply\count255 by-60\relax
\advance\count255 by\time
\edef\thetime{\thetime\ifnum\count255<10 0\fi\the\count255}
\edef\thedate{\number\day-\ifcase\month\or Jan\or Feb\or Mar\or
             Apr\or May\or Jun\or Jul\or Aug\or Sep\or Oct\or
             Nov\or Dec\fi-\number\year}
\def\timstamp{\hbox to\hsize{\tt\hfil\thedate\hfil\thetime\hfil}}}
\maketimestamp
\fancypagestyle{plain}{%
\fancyhf{} 
\fancyfoot[C]{\small \thepage{} of \pageref{LastPage}}}

\numberwithin{equation}{section}  
\allowdisplaybreaks[4]

\newtheorem{theorem}{Theorem}[section]

\newtheorem{lemma}[theorem]{Lemma}
\newtheorem{proposition}[theorem]{Proposition}
\newtheorem{corollary}[theorem]{Corollary}
\theoremstyle{definition}
\newtheorem{definition}[theorem]{Definition} 
\newtheorem{example}{Example}
\theoremstyle{remark}
\newtheorem{remark}{Remark}

\DeclareMathOperator{\supp}{supp} %
\DeclareMathOperator*{\essinf}{ess\,inf} %

\DeclareMathOperator{\exponential}{e}

\newcommandtwoopt{\gaborG}[3][\Lambda][\Gamma]{\mathcal{G}(#3,#1,#2)} 

\newcommand{\myexp}[1]{\exponential^{#1}}

\newcommand{\tran}[1][k]{T_{#1}} 

\newcommand*{\numbersys}[1]{\ensuremath{\mathbb{#1}}}
\newcommand*{\C}{\numbersys{C}}
\newcommand*{\R}{\numbersys{R}}

\newcommand*{\Z}{\numbersys{Z}}

\newcommand*{\N}{\numbersys{N}}

\newcommand*{\cF}{\mathcal{F}}

\newcommand{\itvco}[2]{\ensuremath{\left[{#1},{#2}\right)}} %
\newcommand{\abs}[1]{\ensuremath{\left\lvert#1\right\rvert}}

\newcommand{\norm}[2][]{\ensuremath{\left\lVert#2\right\rVert_{#1}}}

\newcommand{\innerprod}[3][]{\ensuremath{\left\langle #2,#3\right\rangle_{\! #1}}}

\newcommand{\innerprodbig}[3][]{\ensuremath{\bigl \langle #2,#3\bigr\rangle_{\!\!#1}}}

\newcommand{\setprop}[2]{\ensuremath{\left\lbrace{#1} : {#2}\right\rbrace}}

\usepackage{xspace}

\newlength{\dhatheight}

\newcommand*\oline[1]{%
  \vbox{%
    \hrule height 0.5pt
    \kern0.25ex
    \hbox{%
      \kern-0.1em
      \ifmmode#1\else\ensuremath{#1}\fi
      \kern-0.1em
    }
  }
}

{\nopagebreak\par\noindent\hbox to \hsize{\hrulefill}\par\noindent}

\usepackage[backref]{hyperref}
\hypersetup{
 pdfview={FitH},
 pdfstartview={FitH},
 pdfauthor = {Marcin Bownik, Mads Sielemann Jakobsen, Jakob Lemvig, Kasso A.~Okoudjou},
 pdftitle = {On Wilson bases in $L^2(\mathbb{R}^d)$},
 pdfsubject = {harmonic analysis},
 pdfkeywords = {Characterizing equations, frame, Gabor system, orthonormal
   basis, shift-invariant system, Wilson
   system},
 pdfcreator = {LaTeX with hyperref package},
 pdfproducer = PDFlatex}

\makeatletter
\def\blfootnote{\xdef\@thefnmark{}\@footnotetext} 
\def\subjclass{\xdef\@thefnmark{}\@footnotetext}
\long\def\symbolfootnote[#1]#2{\begingroup%
\def\thefootnote{\fnsymbol{footnote}}\footnote[#1]{#2}\endgroup} 
\if@titlepage
  \renewenvironment{abstract}{%
      \titlepage
      \null\vfil
      \@beginparpenalty\@lowpenalty
      \begin{center}%
        \bfseries \abstractname
        \@endparpenalty\@M
      \end{center}}%
     {\par\vfil\null\endtitlepage}
\else
  \renewenvironment{abstract}{%
      \if@twocolumn
        \section*{\abstractname}%
      \else
        \small
        \list{}{%
          \settowidth{\labelwidth}{\textbf{\abstractname:}}
          \setlength{\leftmargin}{50pt}
          \setlength{\rightmargin}{50pt}
          \setlength{\itemindent}{\labelwidth}
          \addtolength{\itemindent}{\labelsep}
        }
        \item[\textbf{\abstractname:}]

      \fi}
      {\if@twocolumn\else\endlist\fi}
\fi
\makeatother

\begin{document}

\title{On Wilson bases in $L^2(\R^d)$}

\date{\today}

\thispagestyle{plain}

\author{
Marcin Bownik\footnote{\noindent Department of Mathematics, University of Oregon, Eugene, OR 97403-1222, USA, and Institute of Mathematics, Polish Academy of Sciences, ul. Wita Stwosza 57,
80--952 Gda\'nsk, Poland, E-mail: \mbox{\protect\url{mbownik@uoregon.edu}}} ,
Mads S.\ Jakobsen\footnote{Norwegian University of Science and Technology, Department of Mathematical Sciences, Trondheim, Norway, \mbox{E-mail: \protect\url{mads.jakobsen@ntnu.no}}}
, Jakob Lemvig\footnote{Technical University of Denmark, Department of Applied Mathematics and Computer Science, Matematiktorvet 303B, 2800 Kgs.\ Lyngby, Denmark, \mbox{E-mail: \protect\url{jakle@dtu.dk}}}
, Kasso A.~Okoudjou\footnote{Department of Mathematics, University of Maryland, College Park, MD 20742, USA,
 E-mail: \mbox{\protect\url{kasso@math.umd.edu}}}}

\maketitle 

  \blfootnote{{\it Key  words and phrases.}  Characterizing equations, frame, Gabor system, orthonormal basis, shift-invariant system, Wilson  system}

\begin{abstract}
A Wilson system is a collection of finite linear combinations of time frequency shifts of a square integrable function. In this paper we use the fact that a Wilson system is a shift-invariant system to explore its relationship with Gabor systems. It is well known that, starting from a tight Gabor frame for
$L^{2}(\R)$ with redundancy $2$, one can construct an orthonormal Wilson basis for $L^2(\R)$ whose generator
is well localized in the time-frequency plane. In this paper, we show that one can construct multi-dimensional orthonormal Wilson bases starting from tight Gabor frames of redundancy $2^k$ where $k=1, 2, \hdots, d$. These results generalize most of the known results about the existence of orthonormal Wilson bases. 
\end{abstract}
\section{Introduction}
\label{sec:introduction}

One of the goals in signal processing and time-frequency analysis is to find convenient series expansions of functions in $L^{2}(\R^d)$. Examples of such series expansions include  Gabor (also called Weyl-Heisenberg) frames. In order to describe these systems we introduce the translation operator $T_{\lambda}$ and the modulation operator $M_{\gamma}$:
\[ T_{\lambda} f(x) = f(x-\lambda), \ \ M_{\gamma} f(x) = e^{2\pi i \langle x , \gamma \rangle} f(x), \ \ f\in L^{2}(\R^d), \, \lambda,\gamma\in \R^d.\]
A Gabor system generated by the window function $g\in L^{2}(\R^d)$ is the set of functions given by $\{M_{\gamma}T_{\lambda}g\}_{\lambda\in \Lambda,\gamma\in\Gamma}$, where $\Lambda$ and $\Gamma$ are lattices in $\R^{d}$.
 Since modulation is a translation in the frequency domain, the
 operation $M_{\gamma}T_{\lambda}$ is called a time-frequency
 shift. Now, a Gabor frame for $L^{2}(\R^{d})$ is a system of the form
 $\{M_{\gamma}T_{\lambda}g\}_{\lambda\in\Lambda,\gamma\in\Gamma}$ for which there exist constants $a,b>0$ such that
\begin{equation}
\label{eq:gabor-frame-def} a \, \Vert f \Vert^{2} \le \sum_{\lambda\in\Lambda,\gamma\in\Gamma} \vert \langle f, M_{\gamma}T_{\lambda}g\rangle \vert^{2} \le b \, \Vert f \Vert^{2} \quad \text{for all} \quad f\in L^{2}(\R^{d}).
\end{equation}
In case
$\{M_{\gamma}T_{\lambda}g\}_{\lambda\in\Lambda,\gamma\in\Gamma}$
satisfies \eqref{eq:gabor-frame-def},  there exists a function $h\in L^{2}(\R^{d})$
 such that
\begin{equation} \label{eq:gabor-rep-formula}
 f = \sum_{\lambda\in\Lambda,\gamma\in\Gamma} \langle f, M_{\gamma}T_{\lambda}g\rangle M_{\gamma}T_{\lambda} h \ \ \text{for all } f\in L^{2}(\R^{d})
\end{equation}
with unconditionally $L^2$-convergence. Whenever the product of
the volume of the two
fundamental domains of the full-rank lattices $\Lambda$ and
$\Gamma$ is strictly less than one, there exist nice window functions
$g \in L^2(\R)$, e.g., in the Schwartz class or the Feichtinger
algebra, such that the Gabor system
$\{M_{\gamma}T_{\lambda}g\}_{\lambda\in\Lambda,\gamma\in\Gamma}$ is a frame \cite{MR1843717}. 


In many applications in engineering and mathematics, it is desirable, not only to have
smooth and localized generators $g \in L^2(\R^d)$, but also \emph{orthogonal} expansions. However, for Gabor frames this is not possible.
Indeed, the famous Balian-Low
Theorem~\cite{Bali81, Bat88, BenHeiWal94, Dau90, DauJan93, Low85}
states that if a Gabor system is an orthonormal basis or a Riesz basis
for $L^{2}(\R^d)$, then $g$ cannot have rapid decay in both time and
frequency.

Yet, in 1991, Daubechies, Jaffard and Journ\'e~\cite{MR1084973},
inspired by work of K.~G.~Wilson \cite{wi87}, were able to construct an
orthonormal basis of (linear combinations of) 
time-frequency shifts of a univariate function $g\in L^{2}(\R)$ with  good time
and frequency localization. The so-called Wilson systems considered in
\cite{MR1084973} are given as: 
\begin{align*}
\mathcal{W}(g) = \{T_{n}g\}_{n\in\Z} & \cup \{\tfrac{1}{\sqrt{2}} T_{n} (M_{m}+(-1)^m M_{-m})g\}_{n\in\Z,m\in \N} \\
& \cup \{\tfrac{1}{\sqrt{2}} T_{n}T_{1/2} (M_{m}-(-1)^m M_{-m})g\}_{n\in\Z,m\in \N}.
\end{align*} 
From this definition, it is clear that, except from the pure
translations $\{T_{n}g\}_{n\in\Z}$, the Wilson systems produce a
bimodular covering of the frequency line, in the sense that each element of
the system has two peaks in its power  spectrum $\abs{\hat{g}}^2$, assuming the window
function is sufficiently localized in frequency. This should be compared with
the unimodular Gabor system, where each element of
$\{M_{\gamma}T_{\lambda}g\}_{\lambda\in\Lambda,\gamma\in\Gamma}$ has a single
peak in the power spectrum. As the following main result of
\cite{MR1084973} shows, Wilson systems do not suffer from the
restrictions of the Balian-Low Theorem.  
\begin{theorem}[\cite{MR1084973}]\label{wilson-basis-dau} Let $g\in L^{2}(\R)$ be such that $\hat{g}(\omega) = \overline{\hat{g}(\omega)}$ and $\Vert g \Vert_2 = 1$. Then  the Gabor system 
$ \{M_{m}T_{n/2}g\}_{m,n\in\Z}$
is a tight frame for $L^{2}(\R)$ if, and only if, the Wilson system
$\mathcal{W}(g)$ 
is an orthonormal basis for $L^2(\R)$.
\end{theorem}
 
The construction of Wilson bases using   
Theorem~\ref{wilson-basis-dau} can be illustrated by the following examples.
\begin{example}
  \begin{enumerate}[(a)]
\item
    Consider the function $g(x) = \cos(\tfrac{1}{2}\pi
    x)\mathds{1}_{[-1,1]}(x)$. One can easily show, e.g., by Lemma~\ref{le:tq-tight-for-shift-invariant}, that  $\{M_{m/2}T_{n}g\}_{m,n\in\Z}$ is a tight Gabor frame with frame
    bound $A=2$ and $\Vert g \Vert_{2}=1$. Moreover,
    $g(x)=\overline{g(-x)}$ for all $x\in \R$ and
    $\hat{g}(\omega) = \overline{\hat{g}(\omega)}$ for all $\omega\in
    \R$. Thus the Wilson system generated by this function is an
    orthonormal basis for $L^{2}(\R)$.
\item
    Consider the function $g(x) = (\sqrt{-\vert x
      \vert+1})\mathds{1}_{[-1,1]}(x)$. As above one can easily show that
    $\{M_{m/2}T_{n}g\}_{m,n\in\Z}$ is a tight Gabor frame with frame
    bound $A=2$ and $\Vert g \Vert_{2}=1$. Moreover,
    $g(x)=\overline{g(-x)}$ and $\hat{g}(\omega) =
    \overline{\hat{g}(\omega)}$, for all $x,\omega\in \R$. Thus the
    Wilson system generated by this function is a orthonormal basis
    for $L^{2}(\R)$.
  \end{enumerate}
\end{example}

An important point of  Theorem~\ref{wilson-basis-dau}, which is also
illustrated by the above two examples,  is that starting from a tight
Gabor frame with redundancy $2$, it is possible to construct an
orthonormal Wilson basis for $L^2(\R)$ whose generator is well
localized in time and frequency, e.g., the generator can be chosen to
be a Schwartz class function or a $\mathcal{C}^{\infty}$ function with compact support. It easily follows that a tensor product
of this orthonormal Wilson basis will lead an orthonormal basis for
$L^2(\R^d)$. But beyond this method, it is not known how to construct 
orthonormal Wilson bases for $L^2(\R^d)$.

Tensoring Wilson bases to $L^2(\R^d)$ has several undesirable side
effects. Firstly, the basis functions of a tensored Wilson basis are
$2^d$-modular hence they give rise to a $2^d$-modular covering of the
frequency domain, akin to the situation of separable wavelets in
$L^2(\R^d)$. Gabor frames
$\{M_{\gamma}T_{\lambda}g\}_{\lambda\in\Lambda,\gamma\in\Gamma}$ are
unimodular in all dimensions, hence give rise to a unimodular covering
of the frequency domain. Bimodular coverings are in most applications
as good as unimodular coverings, in particular, if the signals of
interest $f\in L^2(\R^d)$ are real-valued. However, $2^d$-modular
tensor coverings have a curse of dimensionality since, e.g., symmetric
peaks of the power spectrum of real-valued signals will leak out to
$2^{d-1}$ other locations in frequency. Secondly, they are associated
with highly redundant Gabor frames of redundancy $2^d$. Thirdly, the
generating function has to, naturally, be a separable function of the
form $g_1(x_1)\dots g_d(x_d)$. 
Our goal in this paper is to construct Wilson
orthonormal bases in higher dimension that do not suffer from these
tensoring artifacts. Using the frame theory of shift-invariant systems
\cite{BenLi-1998,MR1946982,MR1916862,MR1601115,
MR1350650}, we
 construct orthonormal Wilson bases for  $L^2(\R^d)$ starting from
tight Gabor frames of redundancy $2^k$, $k=1,\dots,d$. This provides
us with a
ladder of Wilson
orthonormal bases with $2^k$-modular covering of the frequency domain
with $k$ ranging from $1$ to $d$. When $k=d$ we recover the tensored
Wilson bases, but when $k=1$ we obtain a bimodular Wilson orthonormal
basis for $L^2(\R^d)$. As we will see, the latter
construction of bimodular Wilson orthonormal bases is in many ways
superior over the tensored Wilson system. 

Our results also shed new light on univariate (as well as
multivariate) Wilson systems. We show
that, whenever one of the two is well-defined, the frame operators of Gabor and Wilson
systems are identical up to scalar multiplication. We present the view
that Wilson system share several properties with the adjoint of the Gabor
system. Firstly, Gabor
and Wilson systems satisfy a duality principle: the Gabor system is a
frame if and only if the Wilson system is a Riesz basis, and we
provide frame bounds. Secondly, Wilson systems satisfy a density-type
theorem: if a Wilson system is a frame or a tight frame, then it is
automatically a
Riesz basis or an orthonormal
basis, respectively.   

The rest of the paper is organized as follows.  In Section~\ref{sec:wils-syst-sympl} we recall a number of
elementary facts about the symplectic matrices and their role in time-frequency
analysis. Section~\ref{wilson-red2} presents necessary results from the theory
of shift-invariant systems and concerns bimodular Wilson orthonormal bases for $L^{2}(\R^d)$ constructed from
redundancy $2$ tight Gabor frames. In particular, the main result of this section is
Theorem~\ref{th:Wilson-Rd} which generalizes
Theorem~\ref{wilson-basis-dau}.  Furthermore, even for $d=1$ the
results of Section~\ref{wilson-red2} yield a more general statement
than Theorem~\ref{th:Wilson-known} stated in Section~\ref{sec:wils-syst-sympl}.  Finally, in
Section~\ref{sec:wilson-onb-from}, we consider Wilson orthonormal bases (Riesz bases) generated from tight
(non-tight) Gabor frames of redundancy $2^k$ for $k=0, 1, 2, \hdots, d$. In particular, the main results of
Section~\ref{sec:wilson-onb-from} are Theorem~\ref{th:sep-Wilson-Rd} and Proposition~\ref{pr:sep-Wilson-symplectic}.
Theorem \ref{th:sep-Wilson-Rd} is a generalization of Theorem \ref{th:Wilson-Rd}. However, in order to improve readability and understanding we keep the proof of Theorem \ref{th:Wilson-Rd} as a model case for the more technical Theorem \ref{th:sep-Wilson-Rd}.

\section{Wilson systems and symplectic matrices} 
\label{sec:wils-syst-sympl}
In this section we collect some facts about symplectic matrices, which are needed for the proof of Corollary \ref{cor:symplectic-wilson} and Proposition \ref{pr:sep-Wilson-symplectic}. 

But first, we recall the most general
known result concerning univariate Wilson bases due to Kutyniok and Strohmer~\cite{MR2191772}. Similar results can be found in the paper by Wojdy{\l}{\l}o \cite{wo08-1}. We point out that the lattice used to define the tight Gabor frame in Theorem~\ref{th:Wilson-known}  is the image of $\Z^2$ under a symplectic matrix. 

\begin{theorem}[\cite{MR2191772}] \label{th:Wilson-known} Let $a,b,c>0$ and $g\in L^{2}(\R)$ be given. 
If $ab=1/2$, $\hat{g}(\omega) e^{2\pi i c \omega^2/b} = \overline{\hat{g}(\omega)}$ and $\Vert g \Vert_2 = 1$, then the following assertions are equivalent: 
\begin{enumerate}[(i)]
\item The Gabor system 
$ \{T_{na+mc}M_{mb}g\}_{m,n\in\Z} 
$
is a tight frame for $L^{2}(\R)$.
\item The Wilson system
\begin{align*}
\mathcal{W}(g) = \{T_{2na}g\}_{n\in\Z} & \cup \{\tfrac{1}{\sqrt{2}} T_{2na} (T_{mc}M_{mb}+(-1)^m T_{-mc}M_{-mb})g\}_{n\in\Z,m\in \N} \\
& \cup \{\tfrac{1}{\sqrt{2}} T_{2na}T_{a} (T_{mc}M_{mb}-(-1)^m T_{-mc}M_{-mb})g\}_{n\in\Z,m\in \N}
\end{align*} 
is an orthonormal basis for $L^2(\R)$.
\end{enumerate}
\end{theorem}

In Theorem~\ref{th:Wilson-known} it is a slight abuse of language to speak of
$\{T_{na+mc}M_{mb}g\}_{m,n\in\Z} $ as a 
Gabor system; however, since it is unitarily equivalent with the Gabor
system $\{M_{mb}T_{na+mc}g\}_{m,n\in\Z}$, these systems share all
frame theoretic properties, and we will not make any distinction
between such systems in the remainder of this paper.


In addition to the translation operator $T_{\lambda}$ and the modulation
operator $M_{\gamma}$ introduced in Section~\ref{sec:introduction}, we define the following
operators on $L^{2}(\R^d)$.
For $C\in \text{GL}_{\R}(d)$, we define the  \emph{dilation} by $C$: \[ D_{C}f(x) = \abs{\det{C}}^{1/2} f(Cx).\]
For a real-valued, symmetric, $d\times d$ matrix $M$, we define the \emph{chirp-multiplication} by $M$:
\[ S_{M}f(x) = \myexp{\pi i \langle x, Mx\rangle} f(x).\]
The \emph{Fourier transform} is defined for $f\in L^{1}(\R^d)\cap L^{2}(\R^d)$ by 
\[ \mathcal{F}f(\omega) = \int_{\R^d} f(x) \,  e^{-2\pi i \langle \omega,x \rangle } \, dx,\] 
which extends to all of  $L^{2}(\R^d)$ by density.
One readily shows that all the mentioned operators are unitary operators on $L^{2}(\R^d)$ with
\begin{align*}
& (T_{a})^{-1} = (T_{a})^{*} = T_{-a} \, , \ \ 
& (M_{b})^{-1} =  (M_{b})^{*} = M_{-b} \, , \\ 
& (D_{C})^{-1} = (D_{C})^{*} = D_{C^{-1}} \, , \ \
& (S_{M})^{-1} = (S_{M})^{*} = S_{-M} 
\end{align*}
and $\mathcal{F}^{-1}f(\omega) = \mathcal{F}^*f(\omega) = \mathcal{F}f(-\omega)$ for $f\in L^{2}(\R^d)$.

For $\nu = ( \lambda,\gamma) \in \R^d\times \R^d$, we let $\pi(\nu)$ denote the time-frequency shift
operator $M_{\gamma}T_{\lambda}$. 
It is clear that $\pi(\nu)$ is a unitary operator on $L^2(\R^d)$. 

The Fourier transform, dilation operator and chirp-multiplication operator intertwine with a time-frequency shift $\pi(\nu)$, $\nu\in \R^{d} \times {\R}^{d}$ in the following way:
\begin{align}
\mathcal{F} \, \pi(\nu) & = e^{2\pi i \langle \lambda,\gamma\rangle} \, \pi\big( \big[ \begin{smallmatrix}
0 & I \\ -I & 0 \end{smallmatrix} \big] \nu\big) \, \mathcal{F},
\label{eq:F-transform-timeshift-com} 
\\  
D_{C} \, \pi(\nu) & = \pi\big( \big[ \begin{smallmatrix}
C^{-1} & 0 \\ 0 & C^{\top} \end{smallmatrix} \big] \nu\big) \, D_{C},
 \label{eq:dilation-timeshift-com}
\\
S_{M} \, \pi(\nu) & = e^{-\pi i \langle \lambda,M\lambda\rangle} \, \pi\big( \big[ \begin{smallmatrix}
I & 0 \\ M & I \end{smallmatrix} \big] \nu\big) \, S_{M}. \label{eq:chirp-timeshift-com}
\end{align}
Because of these relations we associate to the Fourier transform,
dilation and chirp multiplication operator the following $2d\times 2d$-matrices:
\begin{equation} \label{eq:associated-matrices} \mathcal{F} \ \longleftrightarrow \ \begin{bmatrix}
0 & I \\ -I & 0 \end{bmatrix}, \ \ D_{C} \ \longleftrightarrow \ \begin{bmatrix}
C^{-1} & 0 \\ 0 & C^{\top}
\end{bmatrix}, \ \ S_{M} \ \longleftrightarrow \ 
\begin{bmatrix}
I & 0 \\
M & I
\end{bmatrix},
\end{equation}
where $I$ is the $d\times d$ identity matrix. The three matrices in \eqref{eq:associated-matrices} play an important role in the theory of symplectic matrices:
\begin{definition} A matrix $A\in \text{GL}_{\R}(2d)$ is a symplectic matrix, if
\[ A^{\top} J A = J, \ \ \text{with} \ J = \begin{bmatrix}
0 & I \\ -I & 0 \end{bmatrix}\]
and $I$ being the $d$-dimensional identity matrix. The set of all symplectic matrices is denoted by $\text{Sp}(d)$.
\end{definition}

\begin{theorem}[\cite{MR2827662,MR983366}] All symplectic matrices can be written as a (non-unique) finite composition of matrices of the form as in \eqref{eq:associated-matrices}.
\end{theorem}

We have that $\text{Sp}(1) = \text{SL}_{\R}(2)$, while  for $d\ge 2$ the symplectic matrices $\text{Sp}(d)$ are a proper subgroup of $\text{SL}_{\R}(2d)$. It is advantageous to write symplectic matrices as block matrices of the form
\[ A = \begin{bmatrix}
K & L \\ Q & R \end{bmatrix}. \]
where $K,L,Q$ and $R$ are real valued, $d\times d$ matrices. One can show that the following statements are equivalent:
\begin{enumerate}[(i)]
\item $A = \begin{bmatrix}
K & L \\ Q & R \end{bmatrix}\in \text{Sp}(d)$,
\item $K^{\top}Q$ and $L^{\top}R$ are symmetric matrices and $K^{\top}R-Q^{\top}L = I$,
\item $KL^{\top}$ and $QR^{\top}$ are symmetric matrices and $KR^{\top}-LQ^{\top} = I$.
\end{enumerate}

 We mention the following important decompositions of symplectic matrices into products of matrices of the form as in \eqref{eq:associated-matrices}.

\begin{example} \label{ex:Sp-decomposition} Let $A = \begin{bmatrix}
K & L \\ Q & R \end{bmatrix} \in \text{Sp}(d)$. 
\begin{enumerate}[(i)]
\item If $K\in \text{GL}_{\R}(d)$, then 
\[ A = \begin{bmatrix}
I & 0 \\ QK^{-1} & I \end{bmatrix} \begin{bmatrix}
K & 0 \\ 0 & (K^{\top})^{-1} \end{bmatrix} \begin{bmatrix}
0 & I \\ -I & 0 \end{bmatrix} \begin{bmatrix}
I & 0 \\ -K^{-1}L & I \end{bmatrix} \begin{bmatrix}
0 & -I \\ I & 0 \end{bmatrix}.\]
\item If $L\in \text{GL}_{\R}(d)$, then
\[ A = \begin{bmatrix}
I & 0 \\ RL^{-1} & I \end{bmatrix} \begin{bmatrix}
L & 0 \\ 0 & (L^{\top})^{-1} \end{bmatrix} \begin{bmatrix}
0 & I \\ -I & 0 \end{bmatrix}\begin{bmatrix}
I & 0 \\ L^{-1}K & I \end{bmatrix}.\]
\item If $Q \in \text{GL}_{\R}(d)$, then 
\[ A= \begin{bmatrix}
0 & -I \\ I & 0 \end{bmatrix}\begin{bmatrix}
I & 0 \\ -KQ^{-1} & I \end{bmatrix} \begin{bmatrix}
Q & 0 \\ 0 & (C^{\top})^{-1} \end{bmatrix} \begin{bmatrix}
0 & I \\ -I & 0 \end{bmatrix} \begin{bmatrix}
I & 0 \\ -Q^{-1}R & I \end{bmatrix} \begin{bmatrix}
0 & -I \\ I & 0 \end{bmatrix}.\]
\item If $R\in \text{GL}_{\R}(d)$, then 
\[ A = \begin{bmatrix}
0 & -I \\ I & 0 \end{bmatrix} \begin{bmatrix}
I & 0 \\ -LR^{-1} & I \end{bmatrix} \begin{bmatrix}
R & 0 \\ 0 & (R^{\top})^{-1} \end{bmatrix} \begin{bmatrix}
0 & I \\ -I & 0 \end{bmatrix}\begin{bmatrix}
I & 0 \\ R^{-1}Q & I \end{bmatrix}.\]
\end{enumerate}
\end{example}
Note that this list of examples does not cover all $A\in \text{Sp}(d)$ as there exist symplectic matrices for which each of their block component $K,L,Q,$ and $R$ has zero determinant. To each matrix $A$ in Example~\ref{ex:Sp-decomposition}, we associate a unitary operator via the relations in \eqref{eq:associated-matrices}.

\begin{example} \label{ex:Sp-association} 
To $A = \begin{bmatrix}
K & L \\ Q & R \end{bmatrix} \in \text{Sp}(d)$ we associate the following operators:
\begin{enumerate}[(i)]
\item If $K\in \text{GL}_{\R}(d)$, then 
\[ A \ \longleftrightarrow \ S_{QK^{-1}} \circ D_{K^{-1}} \circ \mathcal{F} \circ S_{-K^{-1}L} \circ \mathcal{F}^{-1}.\]
\item If $L\in \text{GL}_{\R}(d)$, then
\[ A \ \longleftrightarrow \  S_{RL^{-1}} \circ D_{L^{-1}} \circ \mathcal{F} \circ S_{L^{-1}K} .\]
\item If $Q \in \text{GL}_{\R}(d)$, then 
\[ A \ \longleftrightarrow \  \mathcal{F}^{-1} \circ S_{-KQ^{-1}} \circ D_{Q^{-1}} \circ \mathcal{F} \circ S_{-Q^{-1}R} \circ \mathcal{F}^{-1}.\]
\item If $R\in \text{GL}_{\R}(d)$, then 
\[ A \ \longleftrightarrow \ \mathcal{F}^{-1} \circ S_{-LR^{-1}} \circ D_{R^{-1}} \circ \mathcal{F} \circ S_{R^{-1}Q}. \]
\end{enumerate}
\end{example}

More generally,  given any matrix $A\in \text{Sp}(d)$  there exists  a unitary operator $\mu(A)$ acting on $L^{2}(\R^d)$ such that
\begin{equation} \label{eq:muA-and-pi} \mu(A) \pi(\nu) = \varphi(A,\nu) \cdot \pi(A \nu) \mu(A), \end{equation}
where $\varphi(A, \cdot )$ maps vectors $\nu\in\R^{2d}$ into the complex plane with $\vert \varphi \vert = 1$.
Moreover, $\mu(A)$ can be written as a composition of the Fourier transform, suitable dilations and chirp-multiplications. 
For $A\in \text{Sp}(d)$ as in Example~\ref{ex:Sp-decomposition} an operator $\mu(A)$ that satisfies \eqref{eq:muA-and-pi} is given by the associations as in Example~\ref{ex:Sp-association}. 

It is not generally true that there is  a unique operator $\mu(A)$  such that \eqref{eq:muA-and-pi} holds.  Indeed,  from Examples~\ref{ex:Sp-decomposition} and~\ref{ex:Sp-association}: if multiple block components of $A$ are invertible, then we have several choices of the decomposition of $A$ and several operators $\mu(A)$ that we can associate to $A$ so that \eqref{eq:muA-and-pi} holds.
There is a way to make the choice of $\mu(A)$ unique: one constructs
the so-called metaplectic double cover of the symplectic group. For
our results this is not of interest, and we refer to
\cite{MR2827662,MR983366} for more information on this. 
For our needs it is enough that given $A\in \text{Sp}(d)$ a unitary operator $\mu(A)$ exists such that \eqref{eq:muA-and-pi} holds. In specific examples one can use Examples~\ref{ex:Sp-decomposition} and~\ref{ex:Sp-association} to construct such $\mu(A)$. 

Using the relations between $A\in \text{Sp}(d)$, time-frequency shifts and the unitary operator $\mu(A)$ as expressed in \eqref{eq:muA-and-pi} one can show the following well-known results on Gabor systems:

\begin{lemma} \label{le:symplectic-Gabor-extension} Let $\Delta$ be a subset (e.g., a lattice) in $\R^{2d}$, $g\in L^{2}(\R^d)$ and $A\in \textnormal{Sp}(d)$. If $\mu(A)$ is a unitary operator acting on $L^{2}(\R^d)$ such that \eqref{eq:muA-and-pi} holds, then the Gabor system 
$ \{\pi(\nu)g\}_{\nu\in \Delta}$ is a \emph{[frame, tight frame, Riesz basis, orthonormal basis]}, if and only if, the Gabor system $ \{\pi(A \nu) \mu(A) g\}_{\nu\in \Delta}$ is a \emph{[frame, tight frame, Riesz basis, orthonormal basis]}. Moreover, the \emph{[frame, Riesz]} bounds are preserved.
\end{lemma}

We wish to extend Lemma~\ref{le:symplectic-Gabor-extension} to a more general class of systems which includes the Wilson systems that we consider in Theorem~\ref{th:Wilson-Rd}.
To this end we need the following result.
\begin{lemma} \label{le:symplectic-phase-reflection} Let $\nu\in \R^{2d}$ and $A\in \textnormal{Sp}(d)$ be given. If $\mu(A)$ satisfies \eqref{eq:muA-and-pi}, then
\[ \mu(A) \pi(-\nu) = \varphi(A,\nu) \cdot \pi(-A\nu)\mu(A).\]
That is, the phase factor $\varphi(A,\nu)$ in \eqref{eq:muA-and-pi} is invariant under the reflection $\nu\mapsto -\nu$ for all $\nu\in \R^{2d}$.
\end{lemma}
\begin{proof} As $\mu(A)$ can be written as a composition of the Fourier transform, dilations and chirp-mulitplication it is sufficient to prove the result for these three operators. Indeed, for $C\in \text{GL}_{\R}(d)$ and $M\in \text{Sym}_{\R}(d)$ we find from \eqref{eq:F-transform-timeshift-com}, \eqref{eq:dilation-timeshift-com} and \eqref{eq:chirp-timeshift-com} that
\begin{align*}
\mathcal{F} \, \pi(-\nu) & = e^{2\pi i \langle -\lambda,-\gamma\rangle} \, \pi\big({\textstyle -}\big[ \begin{smallmatrix}
0 & I \\ -I & 0 \end{smallmatrix} \big] \nu \big) \, \mathcal{F} = e^{2\pi i \langle \lambda,\gamma\rangle} \, \pi\big({\textstyle -}\big[ \begin{smallmatrix}
0 & I \\ -I & 0 \end{smallmatrix} \big] \nu \big) \, \mathcal{F},
\\  
D_{C} \, \pi(-\nu) & = \pi\big({\textstyle -}\big[ \begin{smallmatrix}
C^{-1} & 0 \\ 0 & C^{\top} \end{smallmatrix} \big] \nu \big) \, D_{C},
\\
S_{M} \, \pi(-\nu) & = e^{-\pi i \langle -\lambda,-M\lambda\rangle} \, \pi\big({\textstyle -}\big[ \begin{smallmatrix}
I & 0 \\ M & I \end{smallmatrix} \big] \nu \big) \, S_{M} = e^{-\pi i \langle \lambda,M\lambda\rangle} \, \pi\big({\textstyle -}\big[ \begin{smallmatrix}
I & 0 \\ M & I \end{smallmatrix} \big] \nu \big) \, S_{M}. 
\end{align*}
In particular, this shows that
\begin{align*}
& \varphi\big( \big[ \begin{smallmatrix}
0 & I \\ -I & 0 \end{smallmatrix} \big], \nu \big)  = \varphi\big( \big[ \begin{smallmatrix}
0 & I \\ -I & 0 \end{smallmatrix} \big], -\nu \big) = e^{2\pi i \langle \lambda,\gamma\rangle}, \\
& \varphi\big( \big[ \begin{smallmatrix}
C^{-1} & 0 \\ 0 & C^{\top} \end{smallmatrix} \big], \nu \big)  = \varphi\big( \big[ \begin{smallmatrix}
C^{-1} & 0 \\ 0 & C^{\top} \end{smallmatrix} \big], - \nu \big) = 1, \\
& \varphi\big( \big[ \begin{smallmatrix}
I & 0 \\ M & I \end{smallmatrix} \big], \nu \big) = \varphi\big( \big[ \begin{smallmatrix}
I & 0 \\ M & I \end{smallmatrix} \big], -\nu \big) = e^{-\pi i \langle \lambda,M\lambda\rangle }.
\end{align*}

\end{proof}

We now immediately have the following extension of Lemma~\ref{le:symplectic-Gabor-extension}:

\begin{lemma} \label{le:symplectic-Wilson-extension}  Let $J$ be an
  index set. For each $j\in J$, let $\Delta_j$ be a subset (e.g., a
  lattice) in $\R^{2d}$ and let $g_j\in L^{2}(\R^d)$. Moreover take
  $A\in \textnormal{Sp}(d)$ and let
  $\{c_{\nu,j}\}_{\nu\in\Delta_{j},j\in J}$ and
  $\{d_{\nu,j}\}_{\nu\in\Delta_{j},j\in J}$ be sequences in
  $\mathbb{C}$. Suppose that $\mu(A)$ is a unitary operator acting on
  $L^{2}(\R^d)$ such that \eqref{eq:muA-and-pi} holds. Then the system 
\[ \bigcup_{j\in J}\big\{\big(c_{\nu,j} \pi(\nu) + d_{\nu,j} \pi(-\nu)\big)g_j\big\}_{\nu\in \Delta_j}\] 
is a \emph{[frame, tight frame, Riesz basis, orthonormal basis]}, if and only if, the system 
\[ \bigcup_{j\in J}\big\{\big(c_{\nu,j} \pi(A\nu) + d_{\nu,j} \pi(-A\nu)\big) \mu(A) g_j\big\}_{\nu\in \Delta_j}\] is a \emph{[frame, tight frame, Riesz basis, orthonormal basis]}. Moreover, the \emph{[frame, Riesz]} bounds are preserved.
\end{lemma}

\section{Bimodular Wilson systems in higher dimensions}\label{wilson-red2}
In this section we consider bimodular Wilson orthonormal bases for $L^2(\R^d)$ that are generated by non-separable
functions $g$. Our main result in this section is
Theorem~\ref{th:Wilson-Rd} stated below. We use boldface
$\mathbf{\frac12}$ to denote the constant vector $(1/2,\dots,1/2) \in
\R^d$, by $\mathbf{\frac12}+\Z^d$ we understand the set
$\setprop{\mathbf{\frac12}+z}{z\in \Z^d}$, and we define  $(-1)^{\vert n \vert} =
(-1)^{n_1+n_2+\ldots+n_d}$ for  vectors $n\in \Z^d$.

%
\begin{theorem} \label{th:Wilson-Rd} Let $g$ be a function in $L^2(\R^d)$ and let
$N$ be a subset of $\Z^d$ such that $N\cap (-N) = \emptyset$ and $N\cup (-N)\cup \{0\} = \Z^d$. 
Consider the Gabor system 
\[ \mathcal{G}(g) = \{ T_{\lambda}M_{\gamma}g\}_{\lambda\in\Z^d\cup (\mathbf{1/2}+\Z^d),\gamma\in\Z^d}\]
and the Wilson system
\begin{align*}
\mathcal{W}(g) =  \{T_{\lambda} g \}_{\lambda\in\Z^d} & \cup \{\tfrac{1}{\sqrt{2}} T_{\lambda}(M_{\gamma}+(-1)^{\vert \gamma \vert} M_{-\gamma})g\}_{\lambda\in\Z^d, \gamma\in N} \\
& \cup \{\tfrac{1}{\sqrt{2}} T_{\lambda+\mathbf{\tfrac{1}{2}}}(M_{\gamma}-(-1)^{\vert \gamma \vert} M_{-\gamma})g\}_{\lambda\in\Z^d,\gamma\in N}.
\end{align*} 
Suppose that $\hat{g}(\omega) = \overline{\hat{g}(\omega)}$. Then the following holds:
\begin{enumerate}[(i)]
\item
 \label{item:1}
 The Gabor system $\mathcal{G}(g)$ is a Bessel sequence with bound $b$ if and only if the Wilson system
  $\mathcal{W}(g)$ is a  Bessel sequence with bound $b/2$. In either (and hence both cases) the Gabor frame operator $S_{\mathcal{G}}$ and the Wilson frame operator $S_{\mathcal{W}}$ satisfy
\[ S_{\mathcal{G}} = 2 S_{\mathcal{W}}.\]
\item The Gabor system $\mathcal{G}(g)$ is a frame with bounds $2a$ and $2b$ for $L^{2}(\R^d)$ if and only if the Wilson system
  $\mathcal{W}(g)$ is a Riesz basis with bounds $a$ and $b$ for $L^{2}(\R^d)$. \label{item:2}
\item The Gabor system $\mathcal{G}(g)$ is a tight frame for $L^{2}(\R^{d})$ with frame bound $a=2$ if and
  only if the Wilson system $\mathcal{W}(g)$ is an orthonormal basis for $L^{2}(\R^{d})$. \label{item:3}
\end{enumerate}
\end{theorem}

    The simple relationship between frame operators of the Gabor system and the Wilson system in Theorem~\ref{th:Wilson-Rd}(\ref{item:1}) seems not have
    been noticed before in the literature, even in dimension one. Indeed, Auscher~\cite{MR1247517} proves a Walnut-type
    representation of an operator $R$ defined as $S_{\mathcal{G}} - 2 S_{\mathcal{W}}$, and
    Gr\"ochenig calls its commutator properties mysterious in \cite{MR1843717}. From Theorem~\ref{th:Wilson-Rd} it
    is now clear that $R$ is in fact the zero operator.  

Statement~(\ref{item:2}) of Theorem~\ref{th:Wilson-Rd} is less
surprising, however, it shows an interesting duality principle akin to
the duality principle of Gabor systems and their adjoint systems. The ``only if''-assertion in~\eqref{item:2} is
Corollary~8.5.6 in \cite{MR1843717} for $d=1$, albeit without bounds. Part~(\ref{item:3}) of Theorem~\ref{th:Wilson-Rd} generalizes
Theorem~\ref{wilson-basis-dau} to higher dimensions in a non-trivial
way. 

In the the following example we show that the standard construction procedure of
``nice'' generators $g$ of univariate Wilson bases, see e.g., \cite{MR1084973} carries over to
bimodular multivariate Wilson bases in Theorem~\ref{th:Wilson-Rd}(iii).
\begin{example}
\label{ex:thm27-nice-generators}
 Take $g \in L^2(\R^d)$ so that
$\mathcal{G}(g)=\{T_{\lambda}M_{\gamma}g\}_{\lambda\in\Z^d\cup (\mathbf{1/2}+\Z^d),\gamma\in\Z^d}$ is a
Bessel system.
If we consider
$\mathcal{G}(g)$ as a
critically sampled, multi-window Gabor system
$\{T_{\lambda}M_{\gamma}g_i\}_{\lambda,\gamma\in\Z^d,i=1,2}$ with two
generators $g_1=g$ and $g_2=(-1)^{\abs{\gamma}}g(\cdot-\mathbf{\tfrac{1}{2}})$, it
follows by \cite[Theorem 8.3.1]{MR1843717} that, for $\alpha \ge 0$,
\begin{equation}
  ZS^\alpha f(x,\omega) = \left(\abs{Zg(x,\omega)}^2 +
    \abs{Zg(x-\mathbf{\tfrac{1}{2}},\omega)}^2\right)^\alpha Zf(x,\omega),\label{eq:Zak-of-S}  
\end{equation}
where $Z$ and $S$ denote the Zak transform and the frame operator,
respectively. Here we tacitly used that $\{T_{\lambda}M_{\gamma}g_2\}_{\lambda,\gamma\in\Z^d}$ and
$\{ T_{\lambda}M_{\gamma}T_{\mathbf{1/2}}\,g\}_{\lambda,\gamma\in\Z^{d}}$ 
have identical frame
operators. If $\mathcal{G}(g)$ is a frame, i.e., if $S_{\mathcal{G}}$ is
invertible, then \eqref{eq:Zak-of-S} also holds for $\alpha<0$.

  From \eqref{eq:Zak-of-S} it is clear that for window functions $g$
  in the Wiener space $W(\R^d)$, the Gabor system
  $\{T_{\lambda}M_{\gamma}g\}_{\lambda\in\Z^d\cup (\mathbf{1/2}+\Z^d),\gamma\in\Z^d}$ is a frame precisely when
  \[
  \essinf_{x,\omega \in \itvco{0}{1}} \left(\abs{Zg(x,\omega)}^2 +
    \abs{Zg(x-\mathbf{\tfrac{1}{2}},\omega)}^2\right) > 0.
  \]
  Let $g \in W(\R^d)$ be such a window function satisfying the
  symmetry condition $\hat g(\omega)=\overline{\hat
    g(\omega)}$. Define $h=S^{-1/2}g=Z^{-1} q Zg$, where
  $q\!=\!\bigl(\abs{Zg}^2 + \abs{ZT_{\mathbf{1/2}}\, g}^2\bigr)^{-1/2} \in\!
  L^\infty(\itvco{0}{1}^2)$. We remark that \eqref{eq:Zak-of-S}
  implies preservation of symmetry under the action of the frame
  operator:
  \begin{equation}
    \label{eq:symmetry}
    \hat{g} \text{ real-valued} \Leftrightarrow   \widehat{Zg} \text{ real-valued}
    \Leftrightarrow   \widehat{S^\alpha g} \text{ real-valued} .
  \end{equation}
  Hence, $\hat h(\omega)=\overline{\hat h(\omega)}$.  Since
  $\{T_{\lambda}M_{\gamma}h\}_{\lambda\in\Z^d\cup (\mathbf{1/2}+\Z^d),\gamma\in\Z^d}$ is a Parseval frame, we conclude, by
  Theorem~\ref{th:Wilson-Rd}, that the Wilson system generated by
  $\sqrt{2}\,h \in W(\R^d)$ (\cite[Theorem 6]{BCKOR},\cite[Corollary 3.1]{KO}) is an orthonormal basis for $L^2(\R^d)$. Note that if
  $g$ is in the Feichtinger algebra $S_0(\R^d)$ or the Schwartz space
  $\mathcal{S}(\R^d)$, then so is $h$, respectively, \cite[Corollary 4.5]{grle04}.
\end{example}

From Lemma \ref{le:symplectic-Wilson-extension} we get the following result.
\begin{corollary} \label{cor:symplectic-wilson} Let $A$ be a matrix in $\textnormal{Sp}(d)$, let $\mu(A)$ be a unitary operator on $L^{2}(\R^{d})$ such that \eqref{eq:muA-and-pi} holds, and let $N$ be a subset of $\Z^{d}$ as in Theorem \ref{th:Wilson-Rd}. For any $g\in L^{2}(\R^{d})$ the symplectic Wilson system 
\begin{align*}
\mathcal{W}_s(g) & =  \big\{\pi(A\lambda) \mu(A) g \big\}_{\lambda\in \Z^d\times\{0\}^d } \\
& \quad \cup \big\{\tfrac{1}{\sqrt{2}} \pi(A\lambda)\big(\pi(A\gamma)+(-1)^{\vert \gamma \vert} \pi(-A\gamma)\big)\mu(A)g\big\}_{\lambda\in\Z^d\times\{0\}^d, \gamma\in\{0\}^d\times N} \\
& \quad \cup \big\{\tfrac{1}{\sqrt{2}} \pi(A\lambda)\pi(A
\lambda^*)\big(\pi(A\gamma)-(-1)^{\vert \gamma \vert}
\pi(-A\gamma)\big) \mu(A)g\big\}_{\lambda\in\Z^d\times\{0\}^d,
  \gamma\in\{0\}^d\times N},
\end{align*}
where $\lambda^{*} = \{ 1/2 \}^{d} \times \{0\}^{d}$, 
is a [frame, Riesz basis, orthonormal basis] if and only if the Wilson system $\mathcal{W}(g)$ in Theorem \ref{th:Wilson-Rd} has the same property. Moreover, the [frame, Riesz] bounds of the two systems are the same.
\end{corollary}

Take $d=1$. If we let $a>0$ be a given positive number and let $c\in \R_{0}^{+}$ be some non-negative number, then we can define the symplectic matrix with associated operator $\mu(A)$ (such that \eqref{eq:muA-and-pi} holds)
\begin{equation}
A = \begin{bmatrix}
2a & c \\ 0 & 1/2a \end{bmatrix} \ \  \text{and} \ \ \mu(A) = D_{1/2a}
\circ \mathcal{F} \circ S_{-c/2a} \circ \mathcal{F}^{-1}.\label{eq:A-as-in-KutStr}
\end{equation}
With these choices Theorem \ref{th:Wilson-Rd}(\ref{item:3}) combined
with Corollary \ref{cor:symplectic-wilson} yields the result from
Kutyniok and Strohmer stated in Theorem \ref{th:Wilson-known}. From
Section~\ref{sec:wils-syst-sympl} it is clear that any matrix $A$ with
determinant one can be used in the construction of symplectic Wilson bases in $\R^{d}$.

The rest of this section is devoted to proving  Theorem~\ref{th:Wilson-Rd}. But first, we need some preliminary results about shift-invariant (SI) systems. The theory presented in Definition \ref{auto}, Lemma \ref{le:tq-tight-for-shift-invariant} and Proposition \ref{thm:wf-function} has been considered specifically for Gabor systems in, e.g., \cite{ja98,MR1460623} and more general, for generalized-shift invariant systems, in \cite{MR1916862,MR2132766}.

\begin{definition}\label{auto}
Let $\Gamma$ be a countable index set and let $\{g_\gamma\}_{\gamma\in \Gamma}\subset L^2(\R^d)$.
For a full-rank lattice $\Lambda=A\Z^d$, where $A \in \textnormal{GL}_{d}(\R)$, the dual lattice or the
\emph{annihilator} is given by $\Lambda^{\perp} =  (A^{-1})^{\top} \Z^d $.
 Suppose that 
 \begin{equation}\label{sfin}
 \sum_{\gamma \in \Gamma} \abs{\hat{g}_\gamma(\omega)}^2< \infty \qquad\text{for a.e. }\omega \in \R^d.
 \end{equation}
For the shift-invariant system $\{T_{\lambda} g_\gamma\}_{\lambda\in \Lambda,\gamma \in \Gamma}$ we define its {\emph autocorrelation} functions $\{t_\alpha\}_{\alpha \in \Lambda^\perp}$ by
\begin{equation}
t_{\alpha}(\omega):=\frac{1}{\abs{\det{A}}} \sum_{\gamma\in \Gamma} \hat{g}_{\gamma}(\omega)
        \overline{\hat{g}_{\gamma}(\omega-\alpha)}  \qquad\text{for a.e. }\omega\in \R^d, \ \alpha \in \Lambda^\perp.\label{eq:t-alpha}
\end{equation}
\end{definition}

By the Cauchy-Schwarz inequality and \eqref{sfin}, the series defining $t_\alpha(\omega)$ are absolutely convergent for a.e. $\omega$. Although the name autocorrelation function is borrowed from signal processing, such functions appear frequently in the study of SI systems. In the case when $\Lambda$ is the standard lattice $\Z^d$, one can employ the characterization of shift-invariant frames in terms of  fiberization operators \cite[Theorem 2.3]{MR1795633} and equivalently by dual Gramians of Ron and Shen \cite{MR1350650}. By  scaling these results hold for shift-invariant systems with respect to an arbitrary (full rank) lattice in $\Lambda \subset\R^d$, see \cite[Section 2.4]{MR2746669}. Indeed, the dual Gramian corresponding to the shift-invariant system $\{T_{\lambda} g_\gamma\}_{\lambda\in \Lambda,\gamma \in \Gamma}$ is the infinite Toeplitz matrix
\[
\tilde G(\omega) = \bigg( \frac{1}{\abs{\det{A}}} \sum_{\gamma\in \Gamma} \hat{g}_{\gamma}(\omega+k)
        \overline{\hat{g}_{\gamma}(\omega+l)} \bigg)_{k,l \in \Lambda^\perp} = (t_{k-l}(\omega+k))_{k,l \in \Lambda^\perp} \qquad \omega\in \R^d.
\]
By \cite[Theorem 2.5]{MR1795633}, $\{T_{\lambda} g_\gamma\}_{\lambda\in \Lambda,\gamma \in \Gamma}$ is a Bessel sequence or frame in $L^2(\R^d)$ with bounds $a$ and $b$ if and only the dual Gramians represent bounded or invertible operators on $\ell^2(\Lambda^\perp)$ with uniform bounds $a$ and $b$ for a.e. $\omega \in \R^d$. In particular, we have the following fact, which has been observed by many authors.

\begin{lemma}[\cite{MR1916862,MR1601115,MR1350650}] 
\label{le:tq-tight-for-shift-invariant} Let $A \in
  \textnormal{GL}_{d}(\R)$, let $\Gamma$ be a countable index set, and let $\{g_\gamma\}_{\gamma\in
    \Gamma}\subset L^2(\R^d)$. Then the following holds:
    \begin{enumerate}[(i)]
    \item   \label{item:5}
    If  $\{T_{\lambda} g_\gamma\}_{\lambda\in A\Z^d,\gamma \in \Gamma}$ is a Bessel sequence with bound $b$, then 
\[
t_0(\omega) =    \sum_{\gamma \in \Gamma} \abs{\hat{g}_\gamma(\omega)}^2 \le b
\qquad\text{ for a.e. }\omega \in \R^d.
\]
\item
$\{T_{\lambda} g_\gamma\}_{\lambda\in A\Z^d,\gamma \in \Gamma}$
is a tight frame for $L^{2}(\R^d)$ with frame bound $a$ if, and only if
      \begin{equation}
        t_{\alpha}(\omega)  =  a
        \delta_{\alpha,0} \qquad\text{for all } \alpha \in (A^{-1})^{\top} \Z^d \text{ and a.e. }\omega\in \R^d.
      \end{equation}
\end{enumerate}
\end{lemma}

For a given function $t \in L^\infty(\R^d)$, define the multiplication operator 
\[
M_t f(x) = t(x) f(x) \qquad\text{for } f\in L^2(\R^d).
\]
For the special choice of $t(x) = e^{ 2 \pi i \langle x, \gamma\rangle}$, $\gamma \in \R^d$, this yields the modulation operator $M_\gamma$, which justifies our notation. Let 
\[
\mathcal{D} = \{ f\in L^2(\R^d): \hat f \in L^\infty(\R^d) \text{ and } \supp \hat f \text{ is bounded} \}.
\]
We will employ the following result, which gives a weak representation of (possibly unbounded) frame operator of the shift-invariant system $\{T_{\lambda} g_\gamma\}_{\lambda\in \Lambda,\gamma \in \Gamma}$ on the dense subspace $\mathcal D \subset L^2(\R^d)$ in terms of autocorrelation functions.

\begin{proposition}[\cite{MR1916862}] \label{thm:wf-function} Let $\Lambda \subset \R^d$ be a full-rank lattice, and let $\Gamma$ be a countable index set. Assume that $\{g_\gamma\}_{\gamma\in
    \Gamma}\subset L^2(\R^d)$ satisfies
  \begin{equation}
    \label{eq:lic}
    \sum_{\gamma \in \Gamma} \abs{\hat{g}_\gamma(\cdot)}^2 \in L^1_{\text{loc}}(\R^d).
  \end{equation}
Let $\{t_\alpha\}_{\alpha \in \Lambda^\perp}$ be the autocorrelation functions of the SI system $\{T_{\lambda} g_\gamma\}_{\lambda\in \Lambda,\gamma \in \Gamma}$.
Then, for any $f\in \mathcal D$, we have
\begin{equation}\label{wf}
\sum_{\gamma \in \Gamma} \sum_{\lambda \in \Lambda} \abs{\innerprod{f}{\tran[\lambda]g_\gamma}}^2= \sum_{\alpha \in \Lambda^\perp} \int_{\mathbb R^d} t_\alpha(\omega) \hat f(\omega-\alpha) \overline{\hat f(\omega)} d\omega = \sum_{\alpha \in \Lambda^\perp} \innerprodbig{M_{t_\alpha}T_\alpha \hat{f}}{\hat{f}}.
\end{equation}
  \end{proposition}
  \begin{proof}
Since the support of $\hat{f}$ is bounded, the sum \eqref{wf} over $\Lambda^\perp=(A^{-1})^{\top}\Z^{d}$ has finitely many non-zero terms. In the proof of \eqref{wf} we shall employ Proposition~2.4 in \cite{MR1916862}, which holds for generalized shift-invariant systems under the local integrability condition (LIC). However, for shift-invariant the LIC used in \cite{MR1916862} is equivalent with~\eqref{eq:lic}. Consequently, for  $f \in \mathcal{D}$,  
   \begin{align}\label{eq:def-of-w-fct}
     w_f(x):= \sum_{ \gamma \in \Gamma} \sum_{\lambda\in \Lambda} \abs{\innerprod{\tran[x]f}{\tran[\lambda]g_\gamma}}^2, \qquad x\in \R^d,
   \end{align}
   is a continuous function that coincides pointwise with the trigonometric polynomial  
    \begin{equation}\label{eq:FS}
     \sum_{\alpha \in \Lambda^\perp} \hat w (\alpha) e^{2 \pi i\langle  \alpha, x \rangle},
 \qquad\text{where }
    \hat w(\alpha)= \int _{\R^d} {t_{\alpha}(\omega)} {\hat f(\omega-\alpha)} \overline{\hat f(\omega)}\,
      d\omega.
    \end{equation}
Taking $x=0$ in \eqref{eq:FS} yields \eqref{wf}.
\end{proof}

\begin{lemma}  \label{le:Wilson-Rd} 
The annihilator $\Lambda^{\perp}$ of the lattice $\Lambda :=  \Z^d \cup \big( \mathbf {1/2} + \Z^d \big) $ is given by 
\[
\Lambda^{\perp} = \{ n\in \Z^{d} : n_{1}+n_{2}+\ldots+n_{d} \in 2 \Z\}.
\]
\end{lemma}
\begin{proof} One easily verifies that $\Lambda$ is a lattice. Define now 
\[ 
H:=\{ n\in \Z^{d} :  n_{1}+n_{2}+\ldots+n_{d} \in 2 \Z\}\]
Take $n\in H$ and $\lambda \in \Lambda$. If $\lambda \in \Z^d$, then $ \langle n,\lambda \rangle \in \Z$.
Likewise, if $\lambda = \big(\mathbf{\tfrac{1}{2}}+ k), k\in \Z^d$, then
\[ \langle n , \lambda \rangle 
= \tfrac{1}{2} \underbrace{(n_1+n_2+\ldots+n_d)}_{\in 2\Z} + \langle n,k\rangle \in \Z. \]
This shows that $H\subset \Lambda^{\perp}$. To show the converse inclusion we observe the following.
By definition we have $\Z^d \subset \Lambda$ and so $\Lambda^{\perp} \subset \Z^d$. Take any $n\in \Lambda^\perp \subset \Z^d$. Then, choosing $\lambda = \mathbf{\tfrac{1}{2}} \in \Lambda$, we have
\[ \langle n , \lambda \rangle = \tfrac{1}{2} (n_1+n_2+\ldots n_d) \in \Z.\] 
Thus, $n \in H$, which shows $H = \Lambda^{\perp}$.
\end{proof}

\begin{lemma} \label{le:wilson-and-gabor-tq} Let $g\in L^{2}(\R^{d})$ and $\hat{g}(\omega) =
  \overline{\hat{g}(\omega)}$. Let $\mathcal{G}(g)$ and $\mathcal{W}(g)$ be the Gabor system and the Wilson
  system considered in Theorem \ref{th:Wilson-Rd}, respectively. Suppose that
   \begin{equation}
    \label{eq:cald-bounded}
     \sum_{\gamma \in \Z^d} \abs{\hat{g}(\omega-\gamma)}^2 < \infty \qquad \text{for a.e. } \omega \in \R^d.
  \end{equation}
Then the following holds:
\begin{enumerate}[(i)]
\item If the Gabor system $\mathcal{G}(g)$ is considered as a shift-invariant system with generators
  $\{M_{\gamma}g\}_{\gamma\in\Z^{d}}$ and with shifts along the lattice $\Lambda=\Z^d\cup(\mathbf{\frac12}+\Z^d)$, then its autocorrelation functions are given by
\[ t_{\alpha,\mathcal{G}}(\omega) = 2\sum_{\gamma\in\Z^d} \hat{g}(\omega-\gamma)\overline{\hat{g}(\omega-\gamma-\alpha)}, \qquad \alpha \in \Lambda^\perp,  \text{ a.e. } \omega\in \R^d.\] 
\item If the Wilson system $\mathcal{W}(g)$ is considered as a shift-invariant system with generators 
\[ g,\  \{\tfrac{1}{\sqrt{2}} (M_{\gamma}+(-1)^{\vert \gamma \vert} M_{-\gamma})g\}_{\gamma\in N} \ \ \text{and}
\ \ \{\tfrac{1}{\sqrt{2}} T_{\mathbf{1/2}}(M_{\gamma}-(-1)^{\vert \gamma \vert} M_{-\gamma})g\}_{\gamma\in N}\] and
with shifts along the lattice $\Z^{d}$, then its autocorrelation functions are given by
\[
 t_{\alpha,\mathcal{W}}(\omega) = \begin{cases}\sum_{\gamma\in\Z^d} \hat{g}(\omega-\gamma)\overline{\hat{g}(\omega-\gamma-\alpha)} & \alpha \in \Lambda^{\perp}, \\ 0 & \alpha\in \Z^{d}\backslash\Lambda^{\perp}, \end{cases} \ \ a.e. \ \omega\in \R^d.\]
\end{enumerate}
\end{lemma}

\begin{proof} First, observe that the assumption \eqref{eq:cald-bounded} guarantees that generators of $\mathcal G(g)$ and $\mathcal W(g)$ satisfy condition \eqref{sfin}. Hence, their autocorrelation functions are well-defined. Then, a straightforward calculation of \eqref{eq:t-alpha} verifies (i). 

The result in (ii) needs some explanation. By Definition \ref{auto}, for $\alpha\in\Z^{d}$ we have
\begin{align} 
 t_{\alpha,\mathcal{W}}(\omega) & = \hat{g}(\omega)\overline{\hat{g}(\omega-\alpha)} \nonumber  \\
  & \quad + \tfrac{1}{2} \sum_{\gamma\in N} \hat{g}(\omega-\gamma) \overline{\hat{g}(\omega-\gamma-\alpha)} + (-1)^{\vert \gamma \vert} \hat{g}(\omega-\gamma) \overline{\hat{g}(\omega+\gamma-\alpha)} \nonumber  \\
& \quad \qquad \qquad + (-1)^{\vert \gamma \vert} \hat{g}(\omega+\gamma) \overline{\hat{g}(\omega-\gamma-\alpha)}+ \hat{g}(\omega+\gamma) \overline{\hat{g}(\omega+\gamma-\alpha)} \label{eq:tq-for-wilson} \\
& \quad + \tfrac{1}{2} e^{2\pi i \langle \mathbf{1/2} ,\alpha\rangle}
 \sum_{\gamma\in N}  \hat{g}(\omega-\gamma) \overline{\hat{g}(\omega-\gamma-\alpha)} - (-1)^{\vert
\gamma \vert } \hat{g}(\omega-\gamma) \overline{\hat{g}(\omega+\gamma-\alpha)} \nonumber  \\
& \qquad \qquad \qquad \qquad \quad  - (-1)^{\vert \gamma \vert } \hat{g}(\omega+\gamma) \overline{\hat{g}(\omega-\gamma-\alpha)}+ \hat{g}(\omega+\gamma) \overline{\hat{g}(\omega+\gamma-\alpha)} .
\nonumber
\end{align}
Note the difference in the signs used in the two sums in the terms
with alternating signs $(-1)^{\abs{\gamma}}$ and the phase factor in front of the second sum.
Because of this phase factor we will consider two cases: (I) $\alpha_{1}+\alpha_{2}+\ldots+\alpha_{d}\in 2\Z$, and (II) $\alpha_{1}+\alpha_{2}+\ldots+\alpha_{d}\in 2\Z+1$. By Lemma \ref{le:Wilson-Rd} these cases correspond to $\alpha\in \Lambda^{\perp}$ and $\alpha\in \Z^{d}\backslash \Lambda^{\perp}$, respectively.
Because of $N\cup (-N) \cup \{0\}=\Z^d$
and $N,-N$ and $\{0\}$ are mutually disjoint sets,   \eqref{eq:tq-for-wilson} yields:
\begin{enumerate}[(I)]
\item for $\alpha\in \Lambda^{\perp}$
  \begin{equation} \label{eq:tq-wilson-I} t_{\alpha,\mathcal{W}}(\omega) =
    \sum_{\gamma\in \Z^d} \hat{g}(\omega-\gamma)
    \overline{\hat{g}(\omega-\gamma-\alpha)}    \qquad\text{a.e. } \omega\in \R^d \, ;
  \end{equation}  
\item for $\alpha\in \Z^{d}\backslash\Lambda^{\perp}$ 
  \begin{equation} \label{eq:tq-wilson-II} t_{\alpha,\mathcal{W}}(\omega) =
    \sum_{\gamma\in \Z^d} (-1)^{\vert \gamma \vert }
    \hat{g}(\omega-\gamma) \overline{\hat{g}(\omega+\gamma-\alpha)} \qquad \text{a.e. } \omega\in \R^d \, .
  \end{equation}  
\end{enumerate}
It remains to show that \eqref{eq:tq-wilson-II} is equal to zero. Take any $\alpha\in \Z^{d}\backslash \Lambda^{\perp}$. By a change of variables
$\gamma\mapsto -\gamma'+\alpha$, we obtain
\begin{equation}  \label{eq:wilson-tqII-with-change-of-variable}
  t_{\alpha,\mathcal{W}}(\omega) = \sum_{\gamma'\in \Z^d} (-1)^{\vert
    (-\gamma'+\alpha)\vert} \hat{g}(\omega+\gamma'-\alpha)
  \overline{\hat{g}(\omega-\gamma')}  
\end{equation}
for a.e. $\omega\in \R^d$.
 For $\alpha\in \Z^d$ with $\alpha_{1}+\alpha_{2}+\ldots+\alpha_{d}\in
 2\Z+1$, we note that
\begin{equation} \label{eq:wilson:sign-change-calc} (-1)^{\vert -\gamma'+\alpha\vert} = (-1)^{\vert- \gamma'\vert} (-1)^{\vert \alpha \vert} = - (-1)^{\vert  \gamma'\vert}.\end{equation}
 Finally, by our assumption $\hat{g}(\omega) = \overline{\hat{g}(\omega)}$, it
 follows that
\begin{equation} \label{eq:real-f-transform}
\hat{g}(\omega+\gamma'-\alpha) \overline{\hat{g}(\omega-\gamma')}=\overline{\hat{g}(\omega+\gamma'-\alpha)} \hat{g}(\omega-\gamma')
\end{equation}
Combining equations \eqref{eq:wilson-tqII-with-change-of-variable}--\eqref{eq:real-f-transform} yields
$t_{\alpha,\mathcal{W}}(\omega)=-t_{\alpha,\mathcal{W}}(\omega)$, hence $t_{\alpha,\mathcal{W}}(\omega)=0$. 
\end{proof}

In the proof of Theorem~\ref{th:Wilson-Rd} we will also need the following two lemmas.

\begin{lemma} \label{le:unit-norm-parseval-frame-is-ONB} Let $\{f_k\}_{k=1}^{\infty}\subset \mathcal{H}$ be a tight frame for $\mathcal H$ with frame bound $a$. 
Then $\{f_k\}_{k=1}^{\infty}$ is an orthonormal basis for $\mathcal
H$, if and only if $\Vert f_k \Vert_{\mathcal{H}} = 1$ for all $k
\in \N$. In this case $a=1$.
\end{lemma}

\begin{lemma}[Theorem 3.5.12 in \cite{MR1601107}]
\label{le:wexraz}
Let $\Delta$ be a lattice in $\R^d \times \R^d$ and let $g\in L^{2}(\R^d)$. If $\{T_{\lambda}M_{\gamma}g\}_{(\lambda,\gamma)\in \Delta}$ is a tight frame, then the set $\{T_{\alpha}M_{\beta}g\}$ of all $(\alpha,\beta)\in \R^d \times \R^d$ for which
\[ (T_{\lambda}M_{\gamma})(T_{\alpha}M_{\beta}) = (T_{\alpha}M_{\beta}) (T_{\lambda}M_{\gamma}) \ \ \text{for all} \ \ (\lambda,\gamma)\in \Delta\]
forms an orthogonal set.
\end{lemma}

We are now ready to prove the main result of this section. 

\begin{proof}[Proof of Theorem~\ref{th:Wilson-Rd}]
We use the setup and notation from Lemma~\ref{le:wilson-and-gabor-tq}. Suppose that either the Gabor system $\mathcal{G}(g)$ or the Wilson system $\mathcal W(g)$ is a Bessel sequence. It follows from Lemma~\ref{le:tq-tight-for-shift-invariant}\eqref{item:5} that $t_{0,\mathcal{G}} \in L^\infty $ or $t_{0,\mathcal{W}}\in L^\infty$, resp. In either case, we have 
\begin{equation}\label{linf}
 \sum_{\gamma \in \Z^d} \abs{\hat{g}(\cdot-\gamma)}^2 \in L^\infty.
\end{equation}
Hence, the assumption \eqref{eq:cald-bounded} in Lemma \ref{le:wilson-and-gabor-tq} holds and we have the following relation between autocorrelation functions
\[ 
  t_{\alpha,\mathcal{W}}(\omega)=
  \begin{cases}
2^{-1} t_{\alpha,\mathcal{G}}(\omega) & \alpha\in \Lambda^\perp, \\
0 & \alpha\in \Z^d \setminus (\Lambda^\perp).
   \end{cases}
\]
By \eqref{linf},  we can apply Proposition~\ref{thm:wf-function} for both $\mathcal{G}(g)$ and $\mathcal{W}(g)$. Hence, for any $f\in\mathcal D$,
\[
 \sum_{\phi \in \mathcal{G}(g)}\abs{\innerprod{f}{\phi}}^2 =  \sum_{\alpha\in \Lambda^\perp} \langle M_{t_{\alpha,\mathcal{G}}} T_{\alpha}  \hat{f}, \hat{f} \rangle  = 2 \sum_{\alpha\in \Z^d} \langle M_{t_{\alpha,\mathcal{W}}} T_{\alpha} \hat{f}, \hat{f} \rangle
 = 2 \sum_{\phi \in \mathcal{W}(g)}\abs{\innerprod{f}{\phi}}^2.
 \]
Now, suppose the Gabor system $\mathcal{G}(g)$ is a Bessel sequence with bound $b$. Then, for any $f\in \mathcal D$,
\[
2 \sum_{\phi \in \mathcal{W}(g)}\abs{\innerprod{f}{\phi}}^2 =  \innerprod{S_{\mathcal{G}}f}{f} \le b \norm{f}^2.
\]
Since $\mathcal{D}$ is dense in $L^2(\R^d)$, this inequality extends to all of $L^2(\R^d)$ which
shows that $\mathcal{W}(g)$ is a Bessel sequence with bound $b/2$ and
\begin{equation}\label{concl}
2 \innerprod{S_{\mathcal{G}}f}{f}  = \innerprod{S_{\mathcal{W}}f}{f}  \qquad\text{for all } f\in L^2(\R^d).
\end{equation}
Since the frame operator is positive and self-adjoint, we obtain $S_{\mathcal{G}} = 2 S_{\mathcal{W}}$. Conversely, assuming that $\mathcal{W}(g)$ is Bessel yields the same conclusion \eqref{concl}, which proves \eqref{item:1}.
It remains to show statements \eqref{item:2} and \eqref{item:3}.

Assume that the Wilson system is a Riesz basis
or an orthonormal basis. Then it is, in particular, also a frame or tight frame,
respectively. However, from the equality $S_{\mathcal{G}} = 2 S_{\mathcal{W}}$, it is clear
that the Gabor system $\mathcal{G}(g)$ is a frame with frame bounds
$a$ and $b$, if and only if the Wilson system $\mathcal{W}(g)$ is a
frame with frame bound $a/2$ and $b/2$.  Hence, it follows that the Gabor system $\mathcal{G}(g)$ is a frame
or tight frame, respectively. 

For the converse directions in statement \eqref{item:2} and \eqref{item:3} we have to work a bit harder. 
We first prove the ``only if''-direction in \eqref{item:3}. Assume therefore that the
Gabor system $\mathcal{G}(g)$ is a tight frame with frame bound $2$,
then, by \eqref{item:1}, the Wilson system is a tight frame with frame bound $1$. By Lemma \ref{le:unit-norm-parseval-frame-is-ONB}, it remains to show that 
\[ \Vert \tfrac{1}{\sqrt{2}} (M_{\gamma} \pm (-1)^{\vert \gamma \vert} M_{-\gamma})g \Vert_{2} = \Vert g \Vert_{2} = 1 \ \ \forall \ \gamma\in N \subset \Z^d.\]
To show this, it suffices to prove that $\{ M_{2\gamma} g\}_{\gamma\in \Z^d}$ is an orthogonal system. By Lemma~\ref{le:wexraz} this is true if the frequency shifts $\{ M_{2\gamma} g\}_{\gamma\in \Z^d}$ commute with the time frequency shifts used in the tight Gabor frame $\mathcal{G}$, i.e., 
\[ 
(M_{2\gamma})(T_{\lambda}M_{\gamma}) =
(T_{\lambda}M_{\gamma})(M_{2\gamma}) \ \ \text{for all} \ \
(\lambda,\gamma)\in \Lambda\times\Gamma,
\] 
where 
\[ 
\Lambda = \Z^d \cup \big( \mathbf{\tfrac{1}{2}} + \Z^d \big) \ \ \text{and} \ \ \Gamma =
\Z^d.
\] 
Indeed, by using the commutator relations $M_{b}T_{a}=e^{2\pi i \langle b,a\rangle} T_{a}M_{b}$, one finds that
\begin{align*}
 (M_{2\gamma})(T_{\lambda}M_{\gamma}) & = e^{2\pi i \langle 2\gamma, \lambda\rangle} (T_{\lambda}M_{\gamma})(M_{2\gamma}) 
\end{align*}
Observe that $\Lambda \subset \tfrac{1}{2} \Z^d$ and thus $2 \Z^d = 2\Gamma \subset \Lambda^{\perp}$. This implies that
indeed 
\[ 
(M_{2\gamma})(T_{\lambda}M_{\gamma}) =
(T_{\lambda}M_{\gamma})(M_{2\gamma}) \ \ \text{for all} \ \
(\lambda,\gamma)\in \Lambda\times\Gamma
\]
 and so all elements in the Wilson system $\mathcal{W}(g)$ have norm
 $1$ and by Lemma~\ref{le:unit-norm-parseval-frame-is-ONB} the system
 $\mathcal{W}(g)$ is a orthonormal basis for $L^{2}(\R^d)$. We have
 now proven \eqref{item:3}.

For the proof of the ``if''-direction in \eqref{item:2} we use the canonical Parseval frame argument as in \cite[Corollary 8.5.6]{MR1843717} which makes use of the result in \eqref{item:2}. More details will be given in the proof of Theorem \ref{th:sep-Wilson-Rd}.
\end{proof}

\section{A scale of Wilson systems} 
\label{sec:wilson-onb-from}

The simplest way of obtaining Wilson bases in higher dimensions is
through tensoring. However, this gives rise to $2^d$-modular covering
of the frequency domain which, as discussed in the introduction, is
often undesirable.  
Theorem~\ref{th:Wilson-Rd} shows that in any dimension one can
construct bimodular Wilson orthonormal bases in $L^{2}(\R^{d})$ from certain tight Gabor frames of redundancy $2$.
In this section we investigate intermediate $2^k$-modular covering of
the frequency domain for $k=1,\dots, d$. 

Let us start by reviewing the tensor construction for $d=2$.

\begin{example}\label{ex2d}
Let $g_{1},g_{2}\in L^{2}(\R)$ be unit norm functions that generate
tight Gabor frames $\{T_{n/2}M_{m}g_{k}\}_{m,n\in\Z}$, $k=1,2$ for
$L^{2}(\R)$. By letting $g(x,y):=g_1(x)g_2(y)$, the Gabor system
$\{T_{n/2}M_{m}g\}_{n\in\Z^2,m\in\Z^2}$ is a tight frame for
$L^{2}(\R^2)$ with density $1/4$, i.e., redundancy 4, and frame bound $4$. Moreover, the tensor product of the two associated one dimensional Wilson systems, which has the rather complicated form \eqref{eq:tensor-wilson}, is an orthonormal basis for $L^2(\R^{2})$.
\begin{equation}\label{eq:tensor-wilson}
\begin{aligned}
& \{ T_{n}g\}_{n\in\Z^2}  
\\
\cup \ & \{ \tfrac{1}{\sqrt{2}}T_{n} (M_{(m_1,0)}+(-1)^{m_1} M_{(-m_{1},0)})g\}_{n\in\Z^2,m_{1}\in\N} 
\\
\cup \ & \{ \tfrac{1}{\sqrt{2}}T_{n}T_{\tfrac{1}{2}(1 , 0)} (M_{(m_1,0)}-(-1)^{m_{1}} M_{(-m_{1},0)})g\}_{n\in\Z^2,m_{1}\in\N}  
\\
\cup \ & \{ \tfrac{1}{\sqrt{2}}T_{n} (M_{(0,m_2)}+(-1)^{m_{2}} M_{(0,-m_{2})})g\}_{n\in\Z^2,m_{2}\in\N} 
 \\
\cup \ & \{ \tfrac{1}{\sqrt{2}}T_{n}T_{\tfrac{1}{2}(0 ,1)} (M_{(0,m_2)}-(-1)^{m_2} M_{(0,-m_{2})})g\}_{n\in\Z^2,m_{2}\in\N} \\
\cup \ & \{ \tfrac{1}{2} T_{n}  (M_{(m_1,m_2)} + (-1)^{m_1} M_{(-m_1,m_2)}  \\
& \qquad \qquad \qquad + (-1)^{m_2} M_{(m_1,-m_2)} + (-1)^{m_1+m_2}M_{-(m_1,m_2)}) g \}_{n\in \Z^2,m\in\N^2} 
\\
\cup \ & \{ \tfrac{1}{2} T_{n} T_{\tfrac{1}{2} (1 , 0)} (M_{(m_1,m_2)} - (-1)^{m_1} M_{(-m_1,m_2)} \\
& \qquad \qquad \qquad + (-1)^{m_2} M_{(m_1,-m_2)} - (-1)^{m_1+m_2}M_{-(m_1,m_2)}) g \}_{n\in \Z^2,m\in\N^2} 
\\
\cup \ & \{ \tfrac{1}{2} T_{n} T_{\tfrac{1}{2} (0 , 1)} (M_{(m_1,m_2)} + (-1)^{m_1} M_{(-m_1,m_2)} 
\\
& \qquad \qquad \qquad  - (-1)^{m_2} M_{(m_1,-m_2)} - (-1)^{m_1+m_2}M_{-(m_1,m_2)}) g \}_{n\in \Z^2,m\in\N^2} 
\\
\cup \ & \{ \tfrac{1}{2} T_{n} T_{\tfrac{1}{2} (1 , 1)} (M_{(m_1,m_2)} - (-1)^{m_1} M_{(-m_1,m_2)} \\
& \qquad \qquad \qquad - (-1)^{m_2} M_{(m_1,-m_2)} + (-1)^{m_1+m_2}M_{-(m_1,m_2)}) g \}_{n\in \Z^2,m\in\N^2} 
\end{aligned}
\end{equation} 
\end{example}

It is  natural to ask if one can generalize this tensor construction allowing a non-separable generator $g$. However, it turns out that the answer to this question is negative. The fact that $g(x,y)=g_{1}(x)g_{2}(y)$ is essential. Indeed, the following example shows that one cannot avoid the separability of $g$.

\begin{example} Consider $\{T_{n/2}M_{m}g\}_{m,n\in\Z^2}$ where  $g \in L^2(\R^2)$ is such that $\hat g=\tfrac{1}{2} \mathds{1}_{D}$, with
\[ D = \{ (x,y)\in \R^{2} \, : \, 0\le x \le 4, 0\le y \le 2, -2+x \le y \le x\}.\] 
Note that $\Vert g \Vert_{2} = 1$. One can easily show that this function generates a tight Gabor frame with density $1/4$ and frame bound $4$. However, the Wilson system in \eqref{eq:tensor-wilson} is not an orthonormal basis. To see this, we apply Lemma~\ref{le:tq-tight-for-shift-invariant} which gives a characterization when the shift-invariant system \eqref{eq:tensor-wilson} is a Parseval frame. In particular, if $\alpha=(1,1)$, then a rather heavy calculation of autocorrelation function of the Wilson system \eqref{eq:tensor-wilson} shows that the necessary condition is that
\[ t_\alpha(\omega)= \sum_{m\in \Z^{2}} (-1)^{m_{1}+m_{2}} \overline{\hat{g}(\omega-m)} \hat{g}(\omega+m-\alpha ) = 0 \qquad\text{a.e. }\omega\in {\R}^2.
\]
However, one finds that
 \[ \sum_{m\in \Z^{2}} (-1)^{m_{1}+m_{2}} \overline{\hat{g}(\omega-m)} \hat{g}(\omega+m-\alpha ) = \frac{1}{2}   \ \ \text{for} \ \omega \in \Omega,\]
where $\Omega = \{ (x,y)\in \R^{2} \, : \, 1\le x \le 4, 1\le y \le 2, -2+x \le y \le x\}$.
Hence, the Wilson system in \eqref{eq:tensor-wilson} with $g$ given as above is not an orthonormal basis for $L^{2}(\R^{2})$.\end{example}

This example suggests that if one assumes that a function $g\in L^{2}(\R^d)$ is separable in all its variables, or more generally separable in the sense of Definition~\ref{sepfunc}, then one can formulate a generalization of Theorem~\ref{th:Wilson-Rd}. In the rest of this section we prove that this is the case. But first, we introduce some necessary concepts.

\begin{definition}\label{refl}
For a vector $\sigma \in \Z^d$ we define the reflection operator
\[ R_{\sigma}: \R^d \to \R^d, \ R_{\sigma}: x\mapsto \big( (-1)^{\sigma_{1}} x_{1}, (-1)^{\sigma_{2}} x_{2}, \ldots, (-1)^{\sigma_{d}} x_{d} \big).\]
On phase-space we define the reflection operator to act by reflecting each component
\[ 
\widetilde{R}_{\sigma} : \R^{2d} \to \R^{2d}, \ \widetilde{R}_{\sigma}: (x,\omega) \mapsto (R_{\sigma}x, R_{\sigma}\omega), \ \ x,\omega\in \R^{d}.
\]
\end{definition}
Clearly, $R_\sigma$ is the identity for $\sigma \in 2\Z^d$. Hence, the reflection operators $R_\sigma$ form a group $\Z^d/(2\Z^d)$, which is identified with its coset representatives $\{0,1\}^d$. For a fixed subgroup $G \subset \Z^d/(2\Z^d)$, we define the orbit of a point $x\in \R^d$ under $G$ to be the set \[ \text{orbit}(x) := \{ y \in \R^d \, : \, y = R_{\sigma}x, \sigma\in G \}.\]

\begin{definition}\label{sepG}
Define the support of  $\sigma \in \Z^d/(2\Z^d)$ by 
\[
\supp \sigma = \{ i \in [d]: \sigma_i \equiv 1 \mod 2 \}, \qquad\text{where }[d]=\{1,\ldots,d\}.
\]
We say that a subgroup $G \subset \Z^d/(2\Z^d)$ is {\it separable} if there exists generators $\sigma^1,\ldots,\sigma^k$ of $G$ such that
\[
\supp \sigma^i \cap \supp \sigma^j = \emptyset \qquad\text{for }i\ne j.
\]
\end{definition}
It follows that  a separable group $G$ is uniquely determined by a collection of non-empty disjoint sets $S_i = \supp \sigma^i \subset [d]$, $i=1,\ldots, k$. In general, the set $S_0= [d] \setminus \bigcup_{i=1}^k S_i$ might be non-empty.

\begin{definition}\label{sepfunc}
For any subset $S=\{s_1<\ldots<s_m\} \subset [d]$, let $P_S: \R^d \to \R^{m}$ be the coordinate projection given by
\[
P_S(x_1,\ldots,x_d) = (x_{s_1},\ldots,x_{s_m}) \qquad x\in \R^d.
\]
We say that a function $g: \R^d \to \C$ is separable with respect to a separable group $G$, if there there exists functions 
$g_i: \R^{|S_i|} \to \C$, $i=0,\ldots,k$, such that
\begin{equation}\label{sep3}
g(x) = \prod_{i=0}^k g_i \circ P_{S_i} (x) \qquad\text{for } x\in \R^d.
\end{equation}
\end{definition}

We also need the following elementary lemma.

\begin{lemma}\label{dual} Let $G \subset \Z^d/(2\Z^d)$ be a separable group as in Definition~\ref{sepG}. Define the lattice 
\[
\Lambda=  \bigcup G = \bigcup_{\sigma\in G} (\sigma+2\Z^d) .
\]
Then $G$ and its dual group $\widehat G$ can be identified as
\begin{equation}\label{dual1}
G \cong \Lambda/(2\Z^d) \qquad\text{and}\qquad
\widehat G \cong \Z^d/(2\Lambda^\perp),
\end{equation}
where $\Lambda^\perp$ is the dual lattice (annihilator) of $\Lambda$. The duality pairing $\langle \cdot,\cdot \rangle_*$ between elements in $\widehat G$ and $G$ is given by
\begin{equation}\label{dual2}
\langle \alpha + 2\Lambda^\perp, \sigma + 2\Z^d \rangle_* = (-1)^{\langle \alpha, \sigma \rangle} 
\qquad\text{for }\alpha\in \Z^d, \sigma \in \Lambda.
\end{equation}

Moreover, $G$ is self-dual and there exists a canonical isomorphism $I:G \to \widehat G$ satisfying
\begin{equation}\label{dual3}
\langle I(\sigma^i),\sigma^j \rangle \equiv \delta_{i,j} \mod 2,
\end{equation}
where $\sigma^i$, $i=1,\ldots,k$, are generators as in Definition~\ref{sepG}. In particular,
\begin{equation}\label{dual4}
\langle I(\sigma),h \rangle \equiv \langle \sigma, I(h) \rangle \mod 2 \qquad\text{for all }\sigma, h \in G.
\end{equation}
\end{lemma}

\begin{proof}
Observe that
\[
2\Z^d \subset \Lambda \subset \Z^d, \qquad
2\Z^d \subset 2\Lambda^\perp \subset \Z^d.
\]
To prove \eqref{dual1}, we can use the following general fact. If $\Gamma_1 \subset \Gamma_2$ are two (full rank) lattices in $\R^d$, then we have group isomorphism
\begin{equation}\label{dual5}
\widehat{\Gamma_2/\Gamma_1} \cong (\Gamma_1)^\perp /(\Gamma_2)^\perp.
\end{equation}
This is a consequence of the duality theorem \cite[Theorem 2.1.2]{MR1038803} since
\[
\widehat{\Gamma_2/\Gamma_1} \cong \operatorname{Ann}(\hat \Gamma_2,\Gamma_1) \cong \operatorname{Ann}(\R^d/( \Gamma_2)^\perp,\Gamma_1) \cong (\Gamma_1)^\perp /(\Gamma_2)^\perp,
\]
where $\operatorname{Ann}(\hat \Gamma_2,\Gamma_1)$ denotes the annihilator of a subgroup $\Gamma_1$ in $\hat \Gamma_2$. Applying the above to $\Gamma_1=2\Z^d$ and $\Gamma_2 =\Lambda$ yields \eqref{dual1}
\[
\widehat {\Lambda/(2\Z^d)} \cong (\tfrac 12\Z^d)/\Lambda^\perp
\cong  \Z^d/(2\Lambda^\perp).
\]

To prove the pairing \eqref{dual2}, note that any $\alpha+2\Lambda^\perp$ defines a character on $G$ by \eqref{dual2}. Since $G$ is assumed to be separable, we can explicitly identify the dual lattice of $\tfrac{1}{2}\Lambda$ as
\begin{equation}\label{dual7}
2\Lambda^\perp = \bigg\{ \alpha \in \Z^d: \sum_{j\in S_i} \alpha_j \equiv 0 \mod 2 \qquad\text{for }i=1,\ldots, k \bigg\}.
\end{equation}
Hence, if $\alpha \not\in 2\Lambda^\perp$, then \eqref{dual2} defines a non-trivial character on $G$. Thus, all characters on $G$ must be of this form.

For every $i=1,\ldots, k$, choose $n_i \in S_i$. Define the mapping $I$ first on generators
\begin{equation}\label{dual9}
I(\sigma^i)=  \delta_{n_i}+2\Lambda^\perp
\qquad\text{for } i=1,\ldots,k,
\end{equation}
and then extend it to a group homomorphism $I:G \to \hat G$. This is well-defined since all non-trivial elements of $\widehat G$ have torsion $2$. To show, that this is an isomorphism take any non-trivial element $\sigma \in G$ of the form $\sigma=\sum_{i=1}^k c_i \sigma^i$, where $c_i=0,1$. Then,
\[
I(\sigma)= \alpha+ 2\Lambda^\perp, \qquad\text{where }
\alpha=\sum_{i=1}^k c_i \delta_{n_i}.
\]
Since $c_i=1$ for some $i$, by \eqref{dual7} $\alpha \not \in 2\Lambda^\perp$. Hence, $I$ is $1-1$ and thus an isomorphism. 

Finally, \eqref{dual3} follows immediately from \eqref{dual9}. Likewise, by \eqref{dual9} we have for any $\sigma,h\in G$,
\[
\langle I(\sigma),h \rangle \equiv \sum_{i=1}^k \sigma_{n_i}h_{n_i} \equiv \langle \sigma, I(h) \rangle \mod 2.
\]
This completes the proof of the lemma.
\end{proof}

In light of Lemma~\ref{dual} we shall slightly abuse the notation by identifying elements of $\widehat G$ with some fixed choice of coset representatives of $\Z^d/(2\Lambda^\perp)$. We are now ready to formulate the main result of this section. 

\begin{theorem} \label{th:sep-Wilson-Rd} 
Let $G \subset \Z^d/(2\Z^d)$ be a separable group with $k$ generators and thus of order $2^{k}$. Furthermore, let $g$ be a function in $L^{2}(\R^{d})$ and
let $N$ be a subset of $\Z^d$ such that
\[ \vert N \cap \textnormal{orbit}(x) \vert = 1 \ \ \forall x \in \Z^d. \]
For  each  $\gamma\in N$, set $c_{\gamma}= 2^{-k} |\textnormal{orbit}(\gamma)|^{1/2}$.
Consider the Gabor system
\[ \mathcal{G}(g, G) = \{T_{\lambda} M_{\gamma} g \}_{\lambda\in \tfrac{1}{2}\Lambda, \gamma\in \Z^d},
\qquad\text{where }\Lambda = \cup_{\sigma\in G} (\sigma + 2 \Z^{d}),
\]
and the Wilson system
\begin{equation} \label{eq:1711b} \mathcal{W}(g,G) = \{ T_{\lambda} T_{\tfrac{1}{2}h} c_{\gamma} \sum_{\sigma\in G} (-1)^{\langle I(h)+\gamma, \sigma\rangle} M_{R_{\sigma}\gamma} g \}_{\lambda\in \Z^d, h\in G, \gamma\in N}.\end{equation}
If $g$ is separable with respect to $G$ and $\hat{g}(\omega) = \overline{\hat{g}(\omega)}$, then the following holds:
\begin{enumerate}[(i)]
\item
\label{item:6}
 The Gabor system $\mathcal{G}(g,G)$ has Bessel bound $b$ if and only if the Wilson system $\mathcal{W}(g,G)$ has Bessel bound $2^{k}b$. In either (and hence both cases) the Gabor frame operator $S_{\mathcal{G}}$ and the Wilson frame operator $S_{\mathcal{W}}$ satisfy
\[ S_{\mathcal{G}} = 2^{k} S_{\mathcal{W}}.\]
\item The Gabor system $\mathcal{G}(g,G)$ is a frame for $L^{2}(\R^{d})$
  with bounds $a$ and $b$, if and only if the Wilson system
  $\mathcal{W}(g,G)$ is a Riesz basis for $L^{2}(\R^{d})$ with bounds
  $2^{-k}a$ and $2^{-k}b$. \label{item:7}
\item The Gabor system $\mathcal{G}(g,G)$ is a tight frame for
  $L^{2}(\R^{d})$ with frame bound $a=2^{k}$ if and only if the Wilson
  system $\mathcal{W}(g,G)$ is an orthonormal basis for
  $L^{2}(\R^{d})$. \label{item:8}
\end{enumerate}
\end{theorem}

\begin{remark}
Note that the Wilson system $\mathcal W(g, G)$ corresponding to the choice of  $G$ being the trivial subgroup is simply a Gabor system $\{T_{\lambda} M_{\gamma} g \}_{\lambda\in \Z^d, \gamma\in \Z^d}$. Hence, the statements of Theorem~\ref{th:sep-Wilson-Rd}\eqref{item:6} are trivial for $k=0$. In contrast, the Wilson system corresponding to the maximal group $G=\Z^d/(2\Z^d)$, where $k=d$, for appropriate choice of $N \subset \Z^d$, is a tensor product of one dimensional Wilson systems as in Example~\ref{ex2d}. Moreover, observe that the choice of a subgroup $G$ generated by $\sigma=(1,\ldots,1)$ yields the same Wilson system as that in Theorem~\ref{th:Wilson-Rd}.
\end{remark}

The following density-type theorem for Wilson system is an easy consequence
of Theorem~\ref{th:sep-Wilson-Rd}.
\begin{corollary}
  If $\mathcal{W}(g,G)$ is a frame for $L^2(\R^d)$ with bounds $a$ and $b$,
  then  $\mathcal{W}(g,G)$ is a Riesz basis for $L^2(\R^d)$ with bounds $a$ and $b$.
\end{corollary}
\begin{proof}
  If $\mathcal{W}(g,G)$ is a frame for $L^2(\R^d)$ with bounds $a$ and $b$,
  then, by Theorem~\ref{th:sep-Wilson-Rd}\eqref{item:6}, so is
  $\mathcal{G}(g,G)$ with bounds $2^{k}a$ and $2^{k}b$. The conclusion now follows from Theorem~\ref{th:sep-Wilson-Rd}\eqref{item:7}.
\end{proof}

Before proceeding with the proof we need to emphasize that some of the functions appearing in the Wilson system \eqref{eq:1711b} are zero. Hence, they should be disregarded due to the cancellation that might happen for some choices of $h\in G $ and $\gamma \in N$. This is a consequence of the following elementary lemma.

\begin{lemma}\label{gcal}
Let $\gamma \in N$ and $h\in G$. Let  $G_\gamma = \{\sigma \in G: R_\sigma \gamma =\gamma\}$ be the stabilizer of $\gamma$. Consider a character $\chi \in \widehat G_\gamma$ given by 
\[
\chi(\sigma) = (-1)^{\langle I(h)+\gamma,\sigma \rangle} \qquad \text{for }\sigma \in G_\gamma.
\]
Then for any $\sigma_0 \in G$, the following sum over a coset of the quotient group $G/G_\gamma$ satisfies
\begin{equation}\label{gcal2}
\sum_{\sigma \in \sigma_0 +G_\gamma} (-1)^{\langle I(h)+\gamma,\sigma \rangle} =
\begin{cases} \pm |G_\gamma| & \chi \equiv 1,\\
0 & \textnormal{otherwise.} 
\end{cases}
\end{equation}
\end{lemma}

\begin{proof}
If $\sigma_0=0$, then formula \eqref{gcal2} follows easily from \cite[Lemma (23.19)]{HeRo} and Lemma~\ref{dual}. The general case $\sigma_0 \in G$ introduces an additional factor $\pm 1$, hence \eqref{gcal2} holds in full generality.
\end{proof}

A symplectic Wilson system is constructed in the following result.

\begin{proposition} \label{pr:sep-Wilson-symplectic} Assume the same setup as in Theorem \ref{th:sep-Wilson-Rd}. If a matrix $A\in\textnormal{Sp}(d)$ is given with associated operator $\mu(A)$ such that \eqref{eq:muA-and-pi} is satisfied, i.e., 
\[ \mu(A) \pi(\nu) = \varphi(A,\nu) \cdot \pi(Av) \mu(A) \ \ \textnormal{for all} \ \ \nu\in \R^{2d}\]
and where $\vert \varphi(A,\nu)\vert = 1$. Then the sympletic Wilson system
\begin{align*} \mathcal{W}_{s}(g,G) = \Big\{ \pi(A\lambda) \pi(A\lambda^{*}_{h}) \,  c_{\gamma} \, \sum_{\sigma\in G} \varphi(A,\widetilde{R}_{\sigma}\gamma)& (-1)^{\langle I(h)+\gamma, \sigma\rangle} \pi(A\widetilde{R}_{\sigma}\gamma) \mu(A)g  \\
& \, : \, \lambda\in \Z^d\times \{0\}^{d}, h\in G, \gamma\in \{0\}^{d} \times N \Big\} \end{align*}
where $\lambda^{*}_{h} = \tfrac{1}{2} h \times \{0\}^{d}$ for $h \in G$, is a [frame, Riesz basis, orthonormal basis] if and only if the Wilson system $\mathcal{W}(g,G)$ in Theorem \ref{th:sep-Wilson-Rd} has the same property. Moreover, the [frame,Riesz] bounds of the two systems are the same.
\end{proposition}
Note that for this statement the phase-factor $\varphi(A,\widetilde{R}_{\sigma}\gamma)$ from the relation \eqref{eq:muA-and-pi} is important. This was not the case for the Wilson system considered in Theorem \ref{th:Wilson-Rd}. If all numbers in the set $\{\varphi(A,\widetilde{R}_{\sigma}\gamma)\}_{\sigma\in G}$ are the same for every fixed $\gamma\in \{0\}^{d}\times N$, then the phase factor can be omitted from the definition of $\mathcal{W}_{s}(g,G)$.

The key part of the proof of Theorem~\ref{th:sep-Wilson-Rd} is contained in the following lemma.

\begin{lemma} \label{le:sep-wilson-and-gabor-tq} Consider the same setup and the same assumptions as in Theorem \ref{th:sep-Wilson-Rd}.
Suppose that \eqref{eq:cald-bounded} holds. Then the following holds:
\begin{enumerate}[(i)]
\item If the Gabor system $\mathcal{G}(g,G)$ is considered as a shift-invariant system with generators
  $\{M_{\gamma}g\}_{\gamma\in\Z^{d}}$ and with shifts along the lattice $(1/2) \Lambda$, then its autocorrelation functions are given by
\[ t_{\alpha,\mathcal{G}}(\omega) = 2^{k}\sum_{\gamma\in\Z^d} \hat{g}(\omega-\gamma)\overline{\hat{g}(\omega-\gamma-\alpha)}, \qquad \alpha \in 2 \Lambda^\perp,  \text{a.e. } \omega\in \R^d.\] 
\item If the Wilson system $\mathcal{W}(g,G)$ is considered a shift-invariant system with generators 
\[ \{ T_{\tfrac{1}{2}h} c_{\gamma} \sum_{\sigma\in G} (-1)^{\langle I(h)+\gamma, \sigma\rangle} M_{R_{\sigma}\gamma} g \}_{ h\in G, \gamma\in N}\] and
with shifts along the lattice $\Z^{d}$, then  then its autocorrelation functions are given by
\[ t_{\alpha,\mathcal{W}}(\omega) = \begin{cases}\sum_{\gamma\in\Z^d} \hat{g}(\omega-\gamma)\overline{\hat{g}(\omega-\gamma-\alpha)} & \alpha \in 2\Lambda^{\perp}, \\ 0 & \alpha\in \Z^{d}\backslash2\Lambda^{\perp}, \end{cases} \ \ a.e. \ \omega\in \R^d.\]
\end{enumerate}
\end{lemma}

\begin{proof}
 The statement of (i) follows immediately from the definition of autocorrelation functions and the observation that the lattice $\Lambda = \cup_{\sigma\in G} (\sigma + 2\Z^{d})$ has density $2^{-k}$.

Consider now the Wilson system as a shift-invariant system along $\Z^d$ with generators 
\begin{equation}\label{char1}
 \psi_{h,\gamma} = T_{\tfrac{1}{2}h} c_{\gamma} \sum_{\sigma\in G} (-1)^{\langle I(h)+\gamma, \sigma\rangle} M_{R_{\sigma}\gamma} g ,  \ \ h\in G, \ \gamma\in N.
\end{equation}
Then,
\begin{equation} \label{eq:2310b} 
t_{\alpha,W}(\omega) = \sum_{h\in G,\gamma\in N}  \hat{\psi}_{h,\gamma}(\omega) \overline{\hat{\psi}_{h,\gamma}(\omega-\alpha)} \qquad\text{for } \alpha\in \Z^d \text{ and a.e. }
\omega\in \R^d.
\end{equation}
The Fourier transform of the generators $\psi_{h,\gamma}$ are given by 
\[ \hat{\psi}_{h,\gamma} = c_{\gamma} (-1)^{\langle h, \cdot \rangle} \sum_{\sigma\in G} (-1)^{\langle I(h)+\gamma,\sigma\rangle} T_{R_{\sigma}\gamma} \hat{g}.\] 
Hence, by \eqref{dual4} the expression \eqref{eq:2310b} becomes the following:
\begin{align*}
t_{\alpha,W}(\omega) & = \sum_{h\in G,\gamma\in N} \vert c_{\gamma} \vert^2 (-1)^{\langle h,\alpha\rangle } 
\sum_{\sigma,\sigma'\in G} (-1)^{\langle I( h)+\gamma,\sigma+\sigma'\rangle} T_{R_{\sigma}\gamma} \hat{g}(\omega) \overline{T_{R_{\sigma'}\gamma}\hat{g}(\omega-\alpha)} \\
& = \sum_{\gamma\in N,\sigma,\sigma'\in G} \vert c_{\gamma} \vert^{2} (-1)^{\langle\gamma,\sigma+\sigma'\rangle} T_{R_{\sigma}\gamma} \hat{g}(\omega) \overline{T_{R_{\sigma'}\gamma}\hat{g}(\omega-\alpha)} \sum_{h\in G} (-1)^{\langle h,\alpha + I(\sigma + \sigma')\rangle}.
\end{align*}
Note that by \cite[Lemma (23.19)]{HeRo} and Lemma~\ref{dual} for any $\alpha\in\Z^d$ we have
\begin{equation}\label{char}
 \sum_{h\in  G} (-1)^{\langle h,\alpha + I(\sigma + \sigma')\rangle} = \begin{cases} 2^k & \text{if } \alpha+I(\sigma+\sigma') \in 2\Lambda^\perp, \\ 0 & \text{otherwise}. \end{cases}
\end{equation} 
For a fixed $\alpha \in \Z^d$, let $\tilde \alpha\in G$ be such that $I(\tilde \alpha)=\alpha+2\Lambda^\perp$. Hence,
\[
\alpha+I(\sigma+\sigma') \in 2\Lambda^\perp \iff I(\tilde \alpha + \sigma+\sigma') \in 2\Lambda^\perp \iff
\tilde\alpha + \sigma+\sigma' \in 2\Z^d.
\]
Using \eqref{char} we continue our calculation to find that
\begin{align*}
& \quad \ t_{\alpha,W}(\omega) \\
& = \sum_{\gamma\in N,\sigma,\sigma'\in G} \vert c_{\gamma} \vert^{2} (-1)^{\langle\gamma,\sigma+\sigma'\rangle} T_{R_{\sigma}\gamma} \hat{g}(\omega) \overline{T_{R_{\sigma'}\gamma}\hat{g}(\omega-\alpha)} 
\begin{cases} 2^k & \text{if } \tilde \alpha+\sigma+\sigma' \in 2\Z^d, \\ 0 & \text{otherwise}. \end{cases}
\\
& = \sum_{\gamma\in N, \sigma\in G} 2^{k} \,  \vert c_{\gamma} \vert^2  (-1)^{\langle \gamma, \tilde \alpha\rangle}
T_{R_{\sigma}\gamma} \hat{g}(\omega) \overline{T_{R_{\sigma+\tilde \alpha}\gamma}\hat{g}(\omega-\alpha)} 
\\
 & = \sum_{\gamma\in N, \sigma\in G} \frac{\vert \textnormal{orbit}(\gamma)\vert}{2^{k}} (-1)^{\langle \gamma, \tilde \alpha\rangle}
\hat{g}(\omega-R_{\sigma}\gamma) \overline{\hat{g}(\omega-\alpha-R_{\sigma+\tilde \alpha}\gamma)} 
\\
& = \sum_{\gamma\in \Z^d} (-1)^{\langle\gamma,\tilde \alpha\rangle} \hat{g}(\omega-\gamma) \overline{\hat{g}(\omega-\alpha-R_{\tilde \alpha}\gamma)}.
\end{align*}
In the penultimate step we used the fact the stabilizer subgroup $G_\gamma = \{\sigma \in G: R_\sigma \gamma =\gamma\}$ has order $|G_\gamma|=|G|/|\textnormal{orbit}(\gamma)|$.
Hence, 
\begin{equation} \label{eq:1611a} 
t_{\alpha,W}(\omega) = \sum_{\gamma\in \Z^d} (-1)^{\langle\gamma,\tilde \alpha\rangle} \hat{g}(\omega-\gamma) \overline{\hat{g}(\omega-\alpha-R_{\tilde \alpha}\gamma)}
\qquad\text{a.e. }\omega \in \R^d \text{ for all }\alpha\in \Z^d. 
\end{equation}
We now consider two cases: (I) $\alpha\in 2\Lambda^\perp$ and (II)
$\alpha \in \Z^d \setminus 2\Lambda^\perp$.

In case (I) we have $\tilde \alpha\in 2\Z^d$, which implies that $R_{\tilde \alpha}$ is the identity, and further $(-1)^{\langle \gamma,\tilde \alpha\rangle} = 1$ for all $\gamma\in \Z^d$. Therefore, \eqref{eq:1611a} becomes
\begin{equation}\label{eq:1}
 t_{\alpha,W}(\omega) = \sum_{\gamma\in \Z^d} \hat{g}(\omega-\gamma) \overline{\hat{g}(\omega-\alpha-\gamma)} 
\quad\text{a.e. }\omega \in \R^d 
\text{ for all }\alpha\in 2\Lambda^\perp.
\end{equation}

Next we consider case (II). Due to the assumption that $g$ is separable with respect to $G$, $g$ is of the form \eqref{sep3}. By the Fubini theorem 
\[
||g||_2= \prod_{j=0}^k ||g_j||_2,
\]
and
\[
\hat g(\omega) = \prod_{j=0}^k \hat g_j \circ P_{S_j} (\omega) \qquad\text{for } \omega\in \R^d.
\]
Hence, we can rewrite \eqref{eq:1611a} as
\begin{equation} \label{eq:1711a}
t_{\alpha,W}(\omega) = \prod_{j=0}^{k} \Bigg( \sum_{\gamma \in \Z^{|S_j|}} (-1)^{\langle\gamma, P_{S_j} \tilde \alpha \rangle } \hat{g}_{j}(P_{S_j}\omega - \gamma ) \overline{\hat{g}_{j}(P_{S_j} \omega-P_{S_j}\alpha-(-1)^{\tilde \alpha_{n_j}}\gamma) }\Bigg).
\end{equation}

Case (II) implies that $\alpha \in \Z^d \setminus 2\Lambda^\perp$ and $\tilde \alpha \in \Lambda \setminus 2\Z^d$. Therefore, there exists $j= 1,2,\ldots,k$ such that $P_{S_j} \tilde \alpha$ has all odd coordinates. By \eqref{dual3} this implies that $|P_{S_j} \alpha|$ is odd. Consider the $j$-th term in the product \eqref{eq:1711a}, i.e., 
\[ C := \sum_{\gamma \in \Z^{|S_j|}} (-1)^{|\gamma|} \hat{g}_{j}(\omega' - \gamma ) \overline{\hat{g}_{j}(\omega'-P_{S_j}\alpha+\gamma) }
\qquad\text{where }\omega'=P_{S_j} \omega \in \R^{|S_j|}.
 \]
We wish to show that $C = 0$ for a.e. $\omega'$. To this end, as in the proof of Theorem~\ref{th:Wilson-Rd}, we make use of a change of variable: $\gamma \mapsto -\gamma' + P_{S_j}\alpha$. This yields that
\begin{align*} C & = \sum_{\gamma'\in \in \Z^{|S_j|}} (-1)^{|-\gamma'+P_{S_j}\alpha|} \hat g_j(\omega'+\gamma'-P_{S_j}\alpha) \overline{\hat g_j (\omega'-\gamma')} \\
& = - \sum_{\gamma' \in \Z^{|S_j|}} (-1)^{|\gamma'|} \hat g_j (\omega'-\gamma) \overline{\hat g_j (\omega'+\gamma-P_{S_j}\alpha)} = - C.
\end{align*}
Here, we used the fact that \eqref{dual3} implies that $|P_{S_j}\alpha|$ is odd and that $\hat g_j (\omega') = \overline{\hat g_j(\omega')}$ for all $\omega'\in \R^{|S_j|}$. We conclude that $C = 0$ and hence we have
\begin{equation}\label{eq:0}
t_{\alpha,W}(\omega)=0 \qquad\text{a.e. }\omega \in \R^d 
\qquad\text{for all }\alpha\in \Z^d \setminus (2\Lambda^\perp).
\end{equation}
This completes the proof of Lemma \ref{le:sep-wilson-and-gabor-tq}.
\end{proof}

We are now ready to give the proof of Theorem~\ref{th:sep-Wilson-Rd}.

\begin{proof} 
Assume that either the Gabor system $\mathcal G(g,G)$ or that the Wilson system $\mathcal W(g,G)$ is a Bessel sequence. Then, the same argument as in the proof of Theorem \ref{th:Wilson-Rd} with the use of Proposition \ref{thm:wf-function} and Lemma \ref{le:sep-wilson-and-gabor-tq} instead of Lemma \ref{le:wilson-and-gabor-tq} shows that for any $f\in\mathcal D$,
\[
 \sum_{\phi \in \mathcal{G}(g,G)}\abs{\innerprod{f}{\phi}}^2 =  \sum_{\alpha\in \Lambda^\perp} \! \langle M_{t_{\alpha,\mathcal{G}}} T_{\alpha}  \hat{f}, \hat{f} \rangle  = 2^k \sum_{\alpha\in \Z^d} \langle M_{t_{\alpha,\mathcal{W}}} T_{\alpha} \hat{f}, \hat{f} \rangle
 = 2^k \!\!\sum_{\phi \in \mathcal{W}(g,G)}\abs{\innerprod{f}{\phi}}^2\!.
 \]
This implies the equality $S_{\mathcal{G}} = 2^{k} S_{\mathcal{W}}$, which shows (i). At the same time it shows the ``if'' direction of (ii) and (iii) as in the proof of Theorem \ref{th:Wilson-Rd}. 

Concerning the converse directions in statements \eqref{item:7} and
\eqref{item:8} we proceed as follows. 
If the Gabor system $\mathcal{G}(g,G)$ is a tight frame with frame bound $2^{k}$, then the Wilson system is a tight frame with frame bound $1$. By Lemma \ref{le:unit-norm-parseval-frame-is-ONB} it remains to show that all non-zero generators \eqref{char1} of the Wilson system \eqref{eq:1711b} have norm equal to $1$.
The assumption that the Gabor system in \eqref{item:6} is a tight frame combined with Lemma~\ref{le:wexraz} imply that the family of functions $\{M_{2\gamma} g \}_{\gamma\in \Lambda^\perp}$ is an orthogonal set. Note that 
\[
R_\sigma \gamma - \gamma \in 2\Z^d \subset 2 \Lambda^\perp \qquad\text{for any } \sigma \in G, \gamma\in\Z^d.
\]
Consequently, for any $\gamma \in N$, the family of functions $\{M_{R_{\sigma}\gamma} g \}_{\sigma \in G}$ is an orthonormal set after neglecting that each function is repeated $|G_\gamma|=|G|/|\text{orbit}(\gamma)|$ times. 
Here, $G_\gamma = \{\sigma \in G: R_\sigma \gamma =\gamma\}$ is the stabilizer of $\gamma$. For a fixed $\gamma \in N$ and $h\in G$, consider the character $\chi \in \widehat G_\gamma$ given as in Lemma~\ref{gcal}.
If $\chi \equiv 1$, then a direct calculation using \eqref{gcal2} shows that
\[
||\psi_{h,\gamma}||^2=|c_\gamma|^2 \, \bigg\| \sum_{\sigma\in G/G_\gamma} \pm |G_\gamma| M_{R_{\sigma}\gamma} g \bigg\|^2 
= |c_\gamma|^2 |G_\gamma|^2 |G/G_\gamma| = 1.
\]
Otherwise, if $\chi \not\equiv 1$, then $\psi_{h,\gamma}=0$ and these generators are vacuous. Therefore, by Lemma~\ref{le:unit-norm-parseval-frame-is-ONB}, the Wilson system
$\mathcal W(g,G)$ is an orthonormal basis of $L^2(\R^d)$.
We have now proven \eqref{item:8}.

To finish the proof of \eqref{item:7} we adapt  the  argument of \cite[Corollary 8.5.6]{MR1843717}: Let $S$ be the frame operator of the Gabor frame $\mathcal{G}(g,G)$. Then, 
\[
S^{-1/2} \mathcal{G}(g) = \mathcal{G}(S^{-1/2}g)\]
 is a Parseval frame for $L^{2}(\R^{d})$. We claim that, just as the function $g$, so is the function $S^{-1/2} g$ separable with respect to group $G$. Since $g$ is separable with respect to $G$, we can write it in the form \eqref{sep3}. Hence, the Gabor system $\mathcal{G}(g,G)$ is a tensor product of the Gabor systems $\{ M_{\gamma}T_{\lambda} g_{j} \, : \, \lambda\in P_{S_j}(\tfrac{1}{2}\Lambda),\gamma\in P_{S_j}(\Z^d)\}$, $j=0,\ldots, k$. Let $T_j$ denote the frame operator of these Gabor systems which acts on $L^2(\R^{|S_j|})$. Hence, the frame operator $S$ is a tensor product of frame operators $T_j$. That is, for any separable function $f\in L^2(\R^d)$ of the form \eqref{sep3} we have
\[
S(f)(x) = \prod_{j=0}^k T_j(f_j)\circ P_{S_j}(x) \qquad\text{for }x\in\R^d.
\]
A similar formula holds for $S^{-1/2}$. Hence, we see that $S^{-1/2}g$ is
separable with respect to $G$. Since each frame operator $T_j$ 
preserves symmetry as in \eqref{eq:symmetry}, it also follows that $\cF
S^{-1/2}g(\omega)=\overline{\cF S^{-1/2}g(\omega)}$. Hence $\mathcal{W}(S^{-1/2}g,G)$ is an orthonormal basis. Moreover, 
\[
\mathcal{W}(S^{-1/2}g,G) = S^{-1/2} \mathcal{W}(g,G).
\]
But this implies that the Wilson system itself is a Riesz basis. This proves \eqref{item:7}.
\end{proof}

\begin{remark}
In general, choosing an arbitrary separable group $G$ of intermediate order $2^k$, $k=1,\ldots,d-1$ leads to a huge number of distinct Wilson systems. Indeed, let $p(n)$ be the partition function that  represents the number of ways of writing $n$ as a sum of positive integers. Then, any partition of $[d]=\{1,\ldots,d\}$ leads to a   separable subgroup $G \subset \Z^d/(2\Z^d)$. Hence, up to a permutation isomorphism there are $p(d)$ distinct separable groups in the dimension $d$. Since $p(d)$ satisfies the asymptotic growth
\[
\log p(d) \sim \pi \sqrt{\frac 23}\sqrt{d} \qquad\text{as } d\to \infty,
\]
hence this number grows rapidly with the dimension $d$.
\end{remark}

By tensoring the construction in
Example~\ref{ex:thm27-nice-generators} and the usual construction of
Wilson bases in dimension one, it is clear that we can construct
generators $g \in L^2(\R^d)$ of $2^k$-modular Wilson bases with good
time-frequency localization for \emph{each} $k=1,\dots,d$. In other
words, for each $k=1,\dots,d$, we can find a subgroup $G$ of order
$2^k$ such that the corresponding Wilson system has nice window
functions generating an orthonormal basis.  
However, not every Wilson system from
Theorem~\ref{th:sep-Wilson-Rd}, i.e., not every subgroup $G$, has nice basis generators.
As an example, consider $d=2$ and take $G$ to be the subgroup with coset
representatives $(0,0)$ and $(1,0)$. Then $\Lambda=\Z \times
2\Z$. Being separable with respect to $G$ means that
$g(x,y)=g_1(x)g_2(y)$. Hence, the Gabor system as in Theorem
\ref{th:sep-Wilson-Rd}\eqref{item:6} with $g(x,y)=g_1(x)g_2(y)$ is a tight frame
for $L^2(\R^2)$ if and only if $\{T_{k/2} M_m g_1\}_{k,m \in \Z}$ and
$\{T_k M_m g_2\}_{k,m \in \Z}$ are tight frames for $L^2(\R)$. However,
by the Balian-Low theorem $g_2$ cannot be well localized in time and
frequency. Hence, the same conclusion holds for $g$. 

While it is now possible by Theorem~\ref{th:sep-Wilson-Rd} to construct
Wilson bases from Gabor frames of redundancy $2^k$, $k=1,\dots,d$, it
is still an open question, mentioned in \cite{MR1843717}, whether other redundancies are
possible. Wojdy{\l}{\l}o \cite{MR2343407} shows that
it is possible to construct redundant Wilson-type tight frames for $L^2(\R)$
from Gabor tight frames of redundancy $3$, however, this approach
does not provide orthogonality. It is our hope that the methods
developed in this paper can be used to attack this long standing open
problem.


\section*{Acknowledgment} 
M.\ Bownik was partially supported by NSF grant DMS-1265711 and by a grant from the Simons Foundation \#426295.
K.~A.~Okoudjou  was partially supported by a grant from the Simons Foundation $\# 319197$ and ARO grant W911NF1610008.


\begin{thebibliography}{10}

\bibitem{MR1247517}
{\sc P.~Auscher}, {\em Remarks on the local {F}ourier bases}, in Wavelets:
  mathematics and applications, Stud. Adv. Math., CRC, Boca Raton, FL, 1994,
  pp.~203--218.

\bibitem{BCKOR}
{\sc R.~Balan, J.~G. Christensen, I.~A. Krishtal, K.~A. Okoudjou, and J.~L.
  Romero}, {\em Multi-window {G}abor frames in amalgam spaces}, Math. Res.
  Lett., 21 (2014), pp.~55--69.

\bibitem{Bali81}
{\sc R.~Balian}, {\em Un principe d'incertitude fort en th\'eorie du signal ou
  en m\'ecanique quantique}, C. R. Acad. Sci. Paris, 292 (1981),
  pp.~1357--1362.

\bibitem{Bat88}
{\sc G.~Battle}, {\em Heisenberg proof of the {B}alian-{L}ow theorem}, Lett.
  Math. Phys., 15 (1988), pp.~175--177.

\bibitem{BenHeiWal94}
{\sc J.~J. Benedetto, C.~Heil, and D.~F. Walnut}, {\em Differentiation and the
  {B}alian-{L}ow theorem}, J. Fourier Anal. Appl., 1 (1995), pp.~355--402.

\bibitem{BenLi-1998}
{\sc J.~J. Benedetto and S.~Li}, {\em The theory of multiresolution analysis
  frames and applications to filter banks}, Appl. Comput. Harmon. Anal., 5
  (1998), pp.~389--427.

\bibitem{MR1795633}
{\sc M.~Bownik}, {\em The structure of shift-invariant subspaces of {$L^2({\bf
  R}^n)$}}, J. Funct. Anal., 177 (2000), pp.~282--309,
  \url{https://doi.org/10.1006/jfan.2000.3635}.

\bibitem{MR2746669}
{\sc M.~Bownik and J.~Lemvig}, {\em Affine and quasi-affine frames for rational
  dilations}, Trans. Amer. Math. Soc., 363 (2011), pp.~1887--1924,
  \url{https://doi.org/10.1090/S0002-9947-2010-05200-6}.

\bibitem{MR1946982}
{\sc O.~Christensen}, {\em An introduction to frames and {R}iesz bases},
  Applied and Numerical Harmonic Analysis, Birkh\"auser Boston Inc., Boston,
  MA, 2003.

\bibitem{Dau90}
{\sc I.~Daubechies}, {\em The wavelet transform, time-frequency localization
  and signal analysis}, IEEE Trans. Inform. Theory, 39 (1990), pp.~961--1005.

\bibitem{MR1084973}
{\sc I.~Daubechies, S.~Jaffard, and J.-L. Journ{\'e}}, {\em A simple {W}ilson
  orthonormal basis with exponential decay}, SIAM J. Math. Anal., 22 (1991),
  pp.~554--573, \url{https://doi.org/10.1137/0522035}.

\bibitem{DauJan93}
{\sc I.~Daubechies and A.~J.~E.~M. Janssen}, {\em Two theorems on lattice
  expansions}, IEEE Trans. Inform. Theory, 39 (1993), pp.~3--6.

\bibitem{MR2827662}
{\sc M.~A. de~Gosson}, {\em Symplectic methods in harmonic analysis and in
  mathematical physics}, vol.~7 of Pseudo-Differential Operators. Theory and
  Applications, Birkh\"auser/Springer Basel AG, Basel, 2011,
  \url{https://doi.org/10.1007/978-3-7643-9992-4}.

\bibitem{MR1601107}
{\sc H.~G. Feichtinger and G.~Zimmermann}, {\em A {B}anach space of test
  functions for {G}abor analysis}, in Gabor analysis and algorithms, Appl.
  Numer. Harmon. Anal., Birkh\"auser Boston, Boston, MA, 1998, pp.~123--170.

\bibitem{MR983366}
{\sc G.~B. Folland}, {\em Harmonic analysis in phase space}, vol.~122 of Annals
  of Mathematics Studies, Princeton University Press, Princeton, NJ, 1989.

\bibitem{MR1843717}
{\sc K.~Gr{\"o}chenig}, {\em Foundations of time-frequency analysis}, Applied
  and Numerical Harmonic Analysis, Birkh\"auser Boston, Inc., Boston, MA, 2001,
  \url{https://doi.org/10.1007/978-1-4612-0003-1}.

\bibitem{grle04}
{\sc K.~{G}r{\"o}chenig and M.~{L}einert}, {\em {W}iener's lemma for twisted
  convolution and {G}abor frames}, {J}. {A}mer. {M}ath. {S}oc., 17 (2004),
  pp.~1--18.

\bibitem{MR1916862}
{\sc E.~Hern{\'a}ndez, D.~Labate, and G.~Weiss}, {\em A unified
  characterization of reproducing systems generated by a finite family. {II}},
  J. Geom. Anal., 12 (2002), pp.~615--662,
  \url{https://doi.org/10.1007/BF02930656}.

\bibitem{HeRo}
{\sc E.~Hewitt and K.~A. Ross}, {\em Abstract harmonic analysis. {V}ol. {I}},
  vol.~115 of Grundlehren der Mathematischen Wissenschaften [Fundamental
  Principles of Mathematical Sciences], Springer-Verlag, Berlin-New York,
  second~ed., 1979.
\newblock Structure of topological groups, integration theory, group
  representations.

\bibitem{MR1601115}
{\sc A.~J. E.~M. Janssen}, {\em The duality condition for {W}eyl-{H}eisenberg
  frames}, in Gabor analysis and algorithms, Appl. Numer. Harmon. Anal.,
  Birkh\"auser Boston, Boston, MA, 1998, pp.~33--84.

\bibitem{ja98}
{\sc A.~J. E.~M. {J}anssen}, {\em {T}he duality condition for
  {W}eyl-{H}eisenberg frames.}, in {G}abor {A}nalysis and {A}lgorithms:
  {T}heory and {A}pplications, H.~G. {F}eichtinger and T.~{S}trohmer, eds.,
  1998, pp.~33--84, 453--488.

\bibitem{KO}
{\sc I.~A. Krishtal and K.~A. Okoudjou}, {\em Invertibility of the {G}abor
  frame operator on the {W}iener amalgam space}, J. Approx. Theory, 153 (2008),
  pp.~212--224.

\bibitem{MR2191772}
{\sc G.~Kutyniok and T.~Strohmer}, {\em Wilson bases for general time-frequency
  lattices}, SIAM J. Math. Anal., 37 (2005), pp.~685--711 (electronic),
  \url{https://doi.org/10.1137/S003614100343723X}.

\bibitem{Low85}
{\sc F.~Low}, {\em Complete sets of wave packets}, in A passion for
  {P}hysics--{E}ssays in {H}onor of {G}eoffrey {C}hew, C.~D. et~al., ed., World
  Scientific, Singapore, 1985, pp.~17--22.

\bibitem{MR1350650}
{\sc A.~Ron and Z.~Shen}, {\em Frames and stable bases for shift-invariant
  subspaces of {$L_2(\mathbf R^d)$}}, Canad. J. Math., 47 (1995),
  pp.~1051--1094, \url{https://doi.org/10.4153/CJM-1995-056-1}.

\bibitem{MR1460623}
{\sc A.~Ron and Z.~Shen}, {\em Weyl-{H}eisenberg frames and {R}iesz bases in
  {$L_2(\bold R^d)$}}, Duke Math. J., 89 (1997), pp.~237--282,
  \url{https://doi.org/10.1215/S0012-7094-97-08913-4}.

\bibitem{MR2132766}
{\sc A.~Ron and Z.~Shen}, {\em Generalized shift-invariant systems}, Constr.
  Approx., 22 (2005), pp.~1--45,
  \url{https://doi.org/10.1007/s00365-004-0563-8}.

\bibitem{MR1038803}
{\sc W.~Rudin}, {\em Fourier analysis on groups}, Wiley Classics Library, John
  Wiley \& Sons Inc., New York, 1990,
  \url{https://doi.org/10.1002/9781118165621}.
\newblock Reprint of the 1962 original, A Wiley-Interscience Publication.

\bibitem{wi87}
{\sc K.~G. {W}ilson}, {\em {G}eneralized {W}annier functions},  (1987).
\newblock unpublished manuscript.

\bibitem{MR2343407}
{\sc P.~Wojdy{\l}{\l}o}, {\em Modified {W}ilson orthonormal bases}, Sampl.
  Theory Signal Image Process., 6 (2007), pp.~223--235.

\bibitem{wo08-1}
{\sc P.~{W}ojdy{\l}{\l}o}, {\em {C}haracterization of {W}ilson systems for
  general lattices}, {I}nt. {J}. {W}avelets {M}ultiresolut. {I}nf. {P}rocess.,
  6 (2008), pp.~305--314.

\end{thebibliography}
\end{document}